\documentclass[a4paper,12pt]{article}
\usepackage{geometry}
\usepackage{amsmath,amssymb,amsfonts,amsthm,graphics}
\usepackage{makeidx}
\usepackage{shuffle}
\usepackage{color}
\usepackage[dvips]{graphicx}
\usepackage{eufrak}
\usepackage{mathrsfs}
\usepackage{graphicx}
\usepackage{float}
\usepackage[all]{xy}
\usepackage{tikz}
\usepackage{url}
\usepackage{pgflibraryarrows}
\usepackage{pgflibrarysnakes}
\newtheorem{thm}{Theorem}[section]

\newtheorem{lem}[thm]{Lemma}

\date{}

\begin{document}
\begin{center}
{\Large Connected order ideals and $P$-partitions}
\end{center}

\begin{center}
Ben P. Zhou \\[8pt]
Center for Combinatorics, LPMC\\
Nankai University\\
Tianjin 300071, P.R. China\\[6pt]
benpzhou@163.com
\end{center}

\begin{abstract}
Given a finite poset $P$, we associate a simple graph denoted by $G_P$ with all connected order ideals of $P$ as vertices, and two vertices are adjacent if and only if they have nonempty intersection and are incomparable with respect to set inclusion. We establish a bijection between  the set of   maximum independent sets of $G_P$ and the set of $P$-forests, introduced by F\'eray and Reiner in their study of the fundamental generating function $F_P(\textbf{x})$ associated with $P$-partitions. Based on this bijection, in the cases when $P$  is naturally labeled we show that $F_P(\textbf{x})$ can factorise,  such that each factor is a summation of rational functions  determined by  maximum independent sets of a connected component of $G_P$.
This approach enables us to give an alternative proof for F\'eray and Reiner's nice formula of $F_P(\textbf{x})$ for the case of $P$ being a naturally labeled forest with duplications.
Another consequence of our result is a product formula to compute the number of linear extensions of $P$.
\end{abstract}


\vspace{5pt}

\section{Introduction}
Throughout this paper, we shall assume that $P$ is a  poset  on  $\{1,2,\ldots,n\}$. We use $\leq_P$
to denote the order relation on $P$ to distinguish from the natural order $\leq $ on integers. We say that $P$ is naturally labeled if $i<j$ whenever $i<_P j$.
A $P$-partition is a map $f$ from $P$ to the set $\mathbb{N}$ of nonnegative integers such that
\begin{itemize}
\item[(1)] if $i<_P j$, then $f(i)\geq f(j)$;

\item[(2)] if $i<_P j$ and $i> j$, then $f(i)>f(j)$.
\end{itemize}
For more information on $P$-partitions, we refer the reader to the book \cite{EC} of  Stanley or  the recent
 survey paper \cite{Gessel} of Gessel.
 Let $\mathscr{A}(P)$ denote the set of  $P$-partitions.
The fundamental generating function $F_P(\textbf{x})$
 associated with $P$-partitions  is defined as
\[F_P(\textbf{x})=\sum_{f\in\mathscr{A}(P)}\textbf{x}^f=\sum_{f\in\mathscr{A}(P)}x_1^{f(1)}x_2^{f(2)}\cdots x_n^{f(n)}.\]
One of the most important problems in the theory of $P$-partitions is to determine explicit expressions for $F_P(\textbf{x})$. The main objective of this paper is to show that for any naturally labeled poset $P$, the  generating function  $F_P(\textbf{x})$  can factorize.

Let us first review some background. The first explicit expression for $F_P(\textbf{x})$ was given by Stanley \cite{Stanley-OrderedStructures}. Recall that a linear extension of $P$ is a permutation $w=w_1w_2\cdots w_n$ on $\{1,2,\ldots,n\}$ such that $i<j$ whenever $w_i<_P w_j$. Let $\mathcal{L}(P)$ be the set of   linear extensions of $P$. For a permutation $w$, write
\[\mathrm{Des}(w)=\left\{i\,|\, 1\leq i\leq n-1, w_i>w_{i+1}\right\}\]
for the  descent set of  $w$. Stanley \cite{Stanley-OrderedStructures} showed that
\begin{equation}\label{Stanley-formula for F_P}
  F_P(\textbf{x})=\sum_{w\in \mathcal{L}(P)}\frac{\prod_{i\in \mathrm{Des}(w)}x_{w_1}x_{w_2}\cdots x_{w_i}}{\prod_{j=1}^{n}\left(1-x_{w_1}x_{w_2}\cdots x_{w_j}\right)}.
\end{equation}

Boussicault, F\'eray,  Lascoux  and  Reiner
 \cite{BoussicaultFerayLascouxReiner} obtained  a similar formula for $F_P(\textbf{x})$ when $P$ is a forest, namely,  every element of $P$ is covered by at most one other element. We say that $j$ is the parent of $i$, if $i$ is covered by $j$ in $P$. Bj\"orner and Wachs \cite{BjonerWachs} defined
the descent set of a forest $P$ as
\begin{equation}\label{descent set of forest}
 \mathrm{Des}(P)=\left\{i\,|\,  \text{ if $j$ is the parent of $i$, then $i>j$}\right\}.
\end{equation}
Thus, if $i\in \mathrm{Des}(P)$, then there exists a node $j\in P$ such that $i<_P j$ but $i>j$. In particular, when a forest $P$ is  naturally labeled, the descent set $\mathrm{Des}(P)$ is empty. For a forest $P$, Boussicault, F\'eray,  Lascoux, and  Reiner's formula is stated as
\begin{equation}\label{formula for F_P when P is a forest}
  F_P(\textbf{x})=\frac{\prod_{i\in \mathrm{Des}(P)}\prod_{k\leq_P i}~x_k}{\prod_{j=1}^n\left(1-\prod_{\ell\leq_P j}~x_{\ell}
  \right)}.
\end{equation}

Furthermore, F\'eray and Reiner \cite{FerayReiner} obtained a nice formula for $F_P(\textbf{x})$ when $P$ is a naturally labeled
forest with duplications, whose definition is given below.
Recall that an order ideal of $P$ is a subset $J$ such that if $i\in J$ and $j\leq_P i$, then $j\in J$.
Throughout the rest of this paper, we will use $J$ to represent an order ideal of $P$.
An order ideal $J$ is connected  if the Hasse diagram of $J$ is a connected graph. A poset $P$ is called a forest  with duplications if for any connected order ideal $J_a$ of $P$, there exists at most one other connected order ideal $J_b$ such that $J_a$ and $J_b$ intersect nontrivially, namely,
\[J_a\cap J_b \neq\emptyset,\ \ \   \ J_a\not\subset J_b\ \ \  \text{and}\ \ \ J_b\not\subset J_a.\]
We would like to point out that a naturally labeled
forest must be a naturally labeled
forest with duplications, while the Hasse diagram of
a naturally labeled
forest with duplications needs not to be a forest.
Let $\mathcal{J}_{conn}(P)$ be the set of  connected order ideals of $P$. For a naturally labeled forest with duplications, F\'eray and Reiner \cite{FerayReiner} proved that
\begin{equation}\label{formula for forests with duplications}
 F_P(\textbf{x})=\frac{\prod_{\{J_a,J_b\}\in \Pi(P)}\left(1-\prod_{i\in J_a}x_i\prod_{j\in J_b}x_j\right)}{\prod_{J\in \mathcal{J}_{conn}(P)}\left(1-\prod_{k\in J}x_k\right)},
\end{equation}
where $\Pi(P)$ consists of all pairs $\{J_a,J_b\}$ of connected order ideals  that
intersect  nontrivially. Note that when $P$ is a naturally labeled forest (with no duplication), both $\mathrm{Des}(P)$ and $\Pi(P)$ are empty,
and each connected order ideal $J$ of $P$ must equal to  $\{\ell\mid  \ell\leq_P j\}$ for some $j\in \{1,2,\ldots,n\}$ and vice versa, and hence formula \eqref{formula for forests with duplications} coincides with formula \eqref{formula for F_P when P is a forest} in this special case.

For any poset $P$, F\'eray and Reiner \cite{FerayReiner}
introduced the notion of $P$-forests and obtained a decomposition of the set $\mathcal{L}(P)$ in terms  of linear extensions of $P$-forests.
Recall that a $P$-forest $F$ is a forest   on $\{1,2,\ldots,n\}$ such that  for any node  $i$, the  subtree  rooted at $i$ is a connected order ideal of $P$, and  that for any two incomparable nodes $i$ and $j$ in the poset $F$, the union of the subtrees rooted  at $i$ and  $j$ is a disconnected order ideal of $P$.
Let $\mathscr{F}(P)$ stand for the set of $P$-forests.
For example, for the poset $P$ in Figure \ref{a typcial example} there are three $P$-forests $F_1, F_2$ and $F_3$.
\begin{figure}[H]
\begin{minipage}{0.4\textwidth}
\begin{center}
   \begin{tikzpicture}[scale=1]
      \draw (3,-1)--(3,0)--(3,1);
      \draw (3,0)--(4,-1)--(4,0);
      \draw  (4,-1)--(4,-2);
      \fill(3,-1) circle(0.06cm); \coordinate[label=below:$3$] (1) at (2.8,-0.5);
      \fill(3,0) circle(0.06cm); \coordinate[label=below:$1$] (3) at (2.8,0.5);
      \fill(3,1) circle(0.06cm); \coordinate[label=below:$2$] (5) at (3,1.6);
      \fill(4,-1) circle(0.06cm); \coordinate[label=below:$4$] (2) at (4.2,-0.5);
      \fill(4,0) circle(0.06cm); \coordinate[label=below:$5$] (4) at (4,0.6);
      \fill(4,-2) circle(0.06cm); \coordinate[label=below:$6$] (4) at (4,-2);
\coordinate[label=below:$P$] (4) at (3,-2.2);
\end{tikzpicture}
\end{center}
\end{minipage}
\begin{minipage}{0.45\textwidth}
\begin{center}
\begin{tabular}{ c c c }

 \begin{tikzpicture}[scale=0.7]
\draw (0,0)--(1,1)--(2,0)--(2,-1) (1,1)--(1,2)--(1,3);
\fill (1,3) circle(0.08cm); \coordinate[label=above:$5$] (1) at (1,3);
\fill (1,2) circle(0.08cm); \coordinate[label=above:$2$] (1) at (1.3,1.8);
\fill (0,0) circle(0.08cm); \coordinate[label=above:$3$] (1) at (0,0);
\fill (1,1) circle(0.08cm); \coordinate[label=above:$1$] (1) at (1.3,0.8);
\fill (2,0) circle(0.08cm); \coordinate[label=above:$4$] (1) at (2.1,0);
\fill (2,-1) circle(0.08cm); \coordinate[label=below:$6$] (1) at (2,-1);
\coordinate[label=below:$F_1$] (4) at (1,-2);
\end{tikzpicture} &
\begin{tikzpicture}[scale=0.7]
\draw (0,0)--(1,1)--(2,0)--(2,-1) (1,1)--(1,2)--(1,3);
\fill (1,3) circle(0.08cm); \coordinate[label=above:$2$] (1) at (1,3);
\fill (1,2) circle(0.08cm); \coordinate[label=above:$5$] (1) at (1.3,1.8);
\fill (0,0) circle(0.08cm); \coordinate[label=above:$3$] (1) at (0,0);
\fill (1,1) circle(0.08cm); \coordinate[label=above:$1$] (1) at (1.3,0.8);
\fill (2,0) circle(0.08cm); \coordinate[label=above:$4$] (1) at (2.1,0);
\fill (2,-1) circle(0.08cm); \coordinate[label=below:$6$] (1) at (2,-1);
\coordinate[label=below:$F_2$] (4) at (1,-2);
\end{tikzpicture}  &  \begin{tikzpicture}[scale=0.7]
\draw (0,0)--(1,1)--(2,0)--(2,-1)--(2,-2) (1,1)--(1,2);
\fill (0,0) circle(0.08cm); \coordinate[label=below:$3$] (1) at (0,0);
\fill (1,1) circle(0.08cm); \coordinate[label=above:$1$] (1) at (1.3,0.8);
\fill (2,0) circle(0.08cm); \coordinate[label=below:$5$] (1) at (2.2,0.7);
\fill (2,-1) circle(0.08cm); \coordinate[label=below:$4$] (1) at (2.2,-0.3);
\fill (2,-2) circle(0.08cm); \coordinate[label=below:$6$] (1) at (2,-2);
\fill (1,2) circle(0.08cm); \coordinate[label=above:$2$] (1) at (1,2);
\coordinate[label=below:$F_3$] (4) at (1,-3);
\end{tikzpicture} \\

\end{tabular}
\end{center}
\end{minipage}
\caption{A poset $P$ and the  corresponding  $P$-forests.}
\label{a typcial example}
\end{figure}
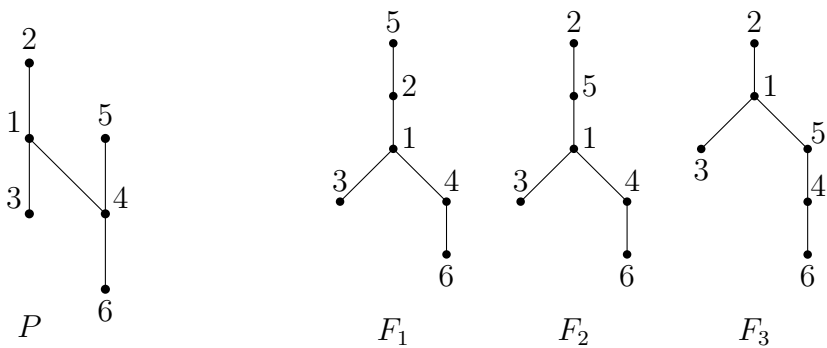
\noindent
F\'eray and Reiner \cite{FerayReiner} showed that
\begin{equation}\label{L decomposition P forests}
  \mathcal{L}(P)=\biguplus_{F\in\mathscr{F}(P)}\mathcal{L}(F),
\end{equation}
which was implied in \cite[Proposition 11.7]{FerayReiner}.
As was remarked by F\'eray and Reiner, the decomposition in \eqref{L decomposition P forests} also appeared in the work  of Postnikov \cite{Postnikov} and  Posnikov, Reiner and Williams \cite{PostnikovReinerWilliams}.
Combining \eqref{Stanley-formula for F_P}, \eqref{formula for F_P when P is a forest} and  \eqref{L decomposition P forests}, one readily sees that
\begin{equation}\label{F_P_FR}
  F_P(\textbf{x})=\sum_{F\in\mathscr{F}(P)}\frac{\prod_{i\in \mathrm{Des}(F)}\prod_{k\leq_{F}i}x_k}{\prod_{j=1}^n\left(1-\prod_{\ell\leq_{F}j} x_{\ell}\right)}.
\end{equation}

Note that both \eqref{Stanley-formula for F_P} and \eqref{F_P_FR} are summation formulas for $F_P(\textbf{x})$. However,
the expression of $F_P(\textbf{x})$ factored nicely for certain posets, as shown in
\eqref{formula for F_P when P is a forest} and \eqref{formula for forests with duplications}. Thus it is desirable to ask that for more general posets $P$ whether $F_P(\textbf{x})$ can factorise.
In this paper, we show that $F_P(\textbf{x})$ can factorise for any naturally labeled poset $P$.

Before stating our result, let  us first introduce some definitions and notations.
In the following we always assume that $P$ is a poset on $\{1,2,\ldots,n\}$.
For any graph $G$, we use $V(G)$ to denote the set of vertices of $G$.
We associate to $P$ a simple graph denoted by $G_P$ with the set $\mathcal{J}_{conn}(P)$ of connected order ideals of $P$  as $V(G_P)$, and two vertices are adjacent if they intersect nontrivially. For example, if $P$ is the  poset  given in Figure \ref{a typcial example}, then $G_P$ is as illustrated  in Figure \ref{posets}, where  we  use $\Lambda_i^P=\{k\,|\, k\leq_P i\}$ to denote the principal order ideal of $P$ generated by $i$, and adopt the  notation $\Lambda_{i,j}^P=\Lambda_i^P\cup\Lambda_j^P$.
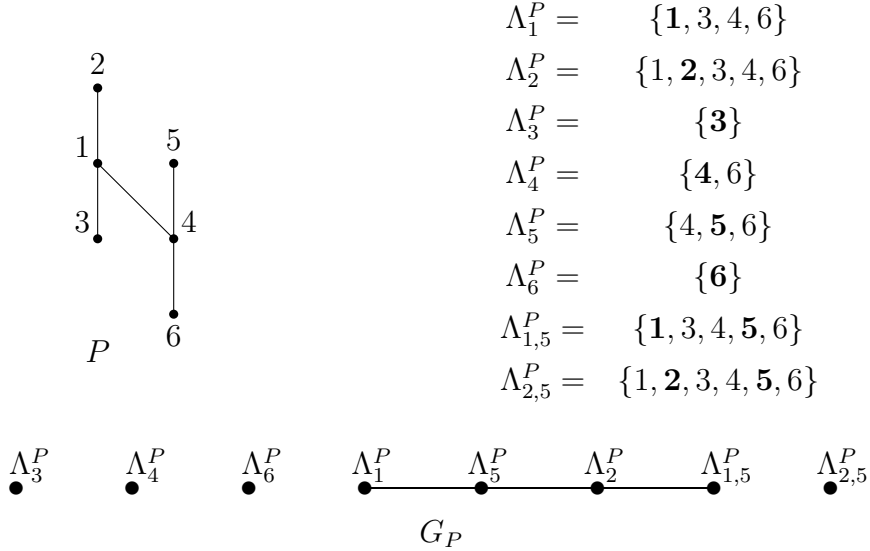
\begin{figure}[h]
\begin{minipage}{0.45\textwidth}
\begin{center}
   \begin{tikzpicture}[scale=1]
      \draw (3,-1)--(3,0)--(3,1);
      \draw (3,0)--(4,-1)--(4,0);
      \draw  (4,-1)--(4,-2);
      \fill(3,-1) circle(0.06cm); \coordinate[label=below:$3$] (1) at (2.8,-0.5);
      \fill(3,0) circle(0.06cm); \coordinate[label=below:$1$] (3) at (2.8,0.5);
      \fill(3,1) circle(0.06cm); \coordinate[label=below:$2$] (5) at (3,1.6);
      \fill(4,-1) circle(0.06cm); \coordinate[label=below:$4$] (2) at (4.2,-0.5);
      \fill(4,0) circle(0.06cm); \coordinate[label=below:$5$] (4) at (4,0.6);
      \fill(4,-2) circle(0.06cm); \coordinate[label=below:$6$] (4) at (4,-2);
\coordinate[label=below:$P$] (4) at (3,-2.2);
\end{tikzpicture}
\end{center}
\end{minipage}
\begin{minipage}{.45\textwidth}
\begin{center}
\begin{tabular}{cc}
$\Lambda_1^P=$&$\{\textbf{1},3,4,6\}$ \\ [5pt]
$\Lambda_2^P=$&$\{1,\textbf{2},3,4,6\}$ \\ [5pt]
$\Lambda_3^P=$&$\{\textbf{3}\}$ \\ [5pt]
$\Lambda_4^P=$&$\{\textbf{4},6\}$ \\ [5pt]
$\Lambda_5^P=$&$\{4,\textbf{5},6\}$  \\ [5pt]
$\Lambda_6^P=$&$\{\textbf{6}\}$  \\ [5pt]
$\Lambda_{1,5}^P=$&$\{\textbf{1},3,4,\textbf{5},6\}$  \\ [5pt]
$\Lambda_{2,5}^P=$&$ \{1,\textbf{2},3,4,\textbf{5},6\}$ \\ [5pt]
\end{tabular}
\end{center}
\end{minipage}
\\
\\
\setlength{\unitlength}{0.9cm}
\begin{picture}(-1,1)
\multiput(2,0)(1.7,0){4}{\circle*{0.2}}
\multiput(8.8,0)(1.7,0){4}{\circle*{0.2}}
\put(7.1,0){\line(1,0){5.1}}
\put(6.9,0.2){$\Lambda_1^P$}

\put(8.6,0.2){$\Lambda_5^P$}

\put(10.3,0.2){$\Lambda_2^P$}

\put(12,0.2){$\Lambda_{1,5}^P$}
\put(1.9,0.2){$\Lambda_3^P$}
\put(3.6,0.2){$\Lambda_4^P$}
\put(5.3,0.2){$\Lambda_6^P$}
\put(13.7,0.2){$\Lambda_{2,5}^P$}
\put(7.9,-0.8){$G_P$}

\end{picture}
\[\]
\caption{Connected order ideals of $P$ and the graph $G_P$.}
\label{posets}
\end{figure}

The first result of this paper is a bijection between the  set of $P$-forests and the set of maximum independent sets of $G_P$. Recall that an independent set of a graph is a subset of vertices such that no two vertices of the subset are adjacent. A maximum independent set of a graph  is an independent set that of largest possible size. For any graph $G$, we use $\mathscr{M}(G)$ to denote the set of  maximum independent sets of $G$. We have the following result.

\begin{thm}\label{The First Main Result}
There exists a bijection between the set $\mathscr{F}(P)$  of $P$-forests and the set  $\mathscr{M}(G_P)$  of maximum independent sets of $G_P$.
\end{thm}

The proof of this result will be given in Section \ref{sect-bijection}, where we establish a bijection $\Phi$ from $\mathscr{F}(P)$ to $\mathscr{M}(G_P)$.
Let $\Psi$ be the inverse map of $\Phi$. In view of the fact that $\Psi(M)$ is a forest, for a maximum independent set $M$ of $G_P$, we can define the descent set $\mathrm{Des}(M)$ of $M$ as the descent set $\mathrm{Des}(\Psi(M))$, namely,
\begin{equation}\label{descent set of M}
  \mathrm{Des}(M)=\mathrm{Des}(\Psi(M)),
\end{equation}
where $\mathrm{Des}(\Psi(M))$ is given by \eqref{descent set of forest}.
Suppose the graph $G_P$ has $h$ connected components, say $C_1,C_2,\ldots,C_h$.
As usual, we use $V(C_r)$ to denote the vertex set of $C_r$ for $1\leq r\leq h$, respectively. It is clear that each maximum independent set of $G_P$ is a disjoint union of maximum independent sets of $G_P$'s connected components.
Let $\mathscr{M}(C_r)$ denote the set of maximum independent sets of $C_r$ for each $1\leq r\leq h$, respectively. 
Given  a $M_r\in \mathscr{M}(C_r)$, we shall further define a descent set for  $M_r$ as illustrated below.
Let  $M$ be a maximum independent set of $G_P$ such that $M\cap V(C_r)=M_r$.
For any $J\in M$, let
\begin{equation}\label{def of M_J}
 \mu(M,J)= \bigcup_{J'\in M,~J'\subset J}J'.
\end{equation}
Define $\mathrm{Des}(M_r,M)$  and $\overline{\mathrm{Des}}(M_r,M)$  as
\begin{eqnarray*}\label{definition of des M_r M}
   \mathrm{Des}(M_r,M)&=&\big\{i\in \mathrm{Des}(M) \mid \{i\}= J\setminus \mu(M,J) \text{ for some } J\in M_r\big\}, \\
   \nonumber \overline{\mathrm{Des}}(M_r,M)&=&\big\{J\in M_r \mid J\setminus \mu(M,J)=\{i\} \text{ for some } i\in \mathrm{Des}(M_r,M)\big\}.
\end{eqnarray*}
It is remarkable that $\mathrm{Des}(M_r,M)$ and $\overline{\mathrm{Des}}(M_r,M)$
are irrelevant to the choice of $M$ when the poset $P$ is naturally labeled. Precisely, we have the following result.

\begin{thm}\label{def-descent}
Suppose that $P$ is a naturally labeled poset and $G_P$ has connected components $C_1,C_2,\ldots,C_h$. Let $M_r$ be a maximum independent set of $C_r$ for some $1\leq r\leq h$. Then for
any two maximum independent sets $M^1,M^2$ of $G_P$ satisfying $M^1\cap V(C_r)=M^2\cap V(C_r)=M_r$,  we have
\begin{eqnarray}
  \mathrm{Des}(M_r,M^1)&=&\mathrm{Des}(M_r,M^2),\\
 \nonumber \overline{\mathrm{Des}}(M_r,M^1)&=&\overline{\mathrm{Des}}(M_r,M^2).
\end{eqnarray}
\end{thm}

Therefore, for a naturally labeled poset $P$ and a given $M_r\in \mathscr{M}(C_r)$ , we can introduce the notation of
$\mathrm{Des}(M_r)$ and $\overline{\mathrm{Des}}(M_r)$, which are respectively defined by
\begin{eqnarray}\label{definition of des overline M_r}
 \mathrm{Des}(M_r) &=& \mathrm{Des}(M_r,M), \\
 \nonumber \overline{\mathrm{Des}}(M_r) &=& \overline{\mathrm{Des}}(M_r,M),
\end{eqnarray}
where $M$ is some maximum independent set of $G_P$ such that $M\cap V(C_r)=M_r$.

The main result of this paper is as follows.

\begin{thm}\label{the seconde main result}
 If $P$ is a naturally labeled poset, and the graph $G_P$ has $h$ connected  components $C_1$, $C_2$,\ldots,$C_h$.  Then we have
\begin{equation}\label{the product formula for F_P}
  F_P(\emph{\textbf{x}})=\prod_{r=1}^h\sum_{ M_r\in \mathscr{M}(C_r)}\frac{\prod_{J\in\overline{\mathrm{Des}}( M_r)  }\prod_{k\in J}x_k}{\prod_{J\in M_r}(1-\prod_{j\in J}x_j)}.
\end{equation}
\end{thm}

This paper is organized as follows.
In Section \ref{sect-bijection}, we shall give a proof of Theorem \ref{The First Main Result}.
In Section \ref{sect-main}, we shall prove Theorems \ref{def-descent} and \ref{the seconde main result}. Based on Theorem \ref{the seconde main result}, we  provide an alternative proof for F\'eray and Reiner's formula \eqref{formula for forests with duplications}. In Section \ref{sect-application}, Theorem \ref{the seconde main result} will be used to derive the generating function of major index of linear extensions of $P$, as well as to count the number of linear extensions of $P$.

\section{The bijection $\Phi$ between $\mathscr{F}(P)$ and $\mathscr{M}(G_P)$} \label{sect-bijection}

The aim of this section is to give a proof of Theorem \ref{The First Main Result}.
To this end, we shall establish a bijection $\Phi$ from $\mathscr{F}(P)$ to $\mathscr{M}(G_P)$ as mentioned before.

To give a description of the map $\Phi$, we first note some properties of $\mathscr{F}(P)$ and $\mathscr{M}(G_P)$. Given $M\in \mathscr{M}(G_P)$ and  $J\in M$, let
\begin{eqnarray}\label{def-UMJ}
  U(M,J) &=& \{J'\in M\mid J'\subset J\}, \\
  \nonumber U_{max}(M,J) &=& \big\{J_a\in U(M,J)\mid  J_a\not\subset J_b \text{ for any } J_b\in U(M,J)\big\}.
\end{eqnarray}
Recall that the set $\mu(M,J)$ is defined in \eqref{def of M_J}, which is also an order ideal of $P$. Thus
\begin{align}\label{eqn-umu}
\mu(M,J)=\bigcup_{J'\in U(M,J)}J'= \bigcup_{J'\in U_{max}(M,J)}J'.
\end{align}
The following assertion will be used in the future proofs.
\begin{lem}\label{lemm-empty}
For any $M\in \mathscr{M}(G_P)$ and  $J\in M$, the intersection of any two elements of $U_{max}(M,J)$ is empty.
\end{lem}

\begin{proof}
Let $J_1,J_2\in U_{max}(M,J)$. Because $U_{max}(M,J)\subset M$ and $M$ is an independent set of $G_P$, it follows that $J_1$ and $J_2$ are not adjacent in $G_P$.  Recall that for any two vertices $J_1,J_2\in \mathcal{J}_{conn}(P)$ of $G_P$, $J_1$ and $J_2$ are not adjacent in $G_P$ if and only if
\[J_1\cap J_2 =\emptyset,\ \ \   \text{or } J_1\subset J_2,\ \ \  \text{or } J_2\subset J_1.\]
On the other hand, by the definition of $U_{max}(M,J)$, there is neither $J_1\subset J_2$ nor $J_2\subset J_1$. Hence $J_a\cap J_b=\emptyset$.
\end{proof}

Given a $P$-forest $F\in\mathscr{F}(P)$, let $\Lambda_i^F=\{j\,|\, j\leq_F i\}$ denote the principal order ideal of $F$ generated by $i$. By definition of $P$-forest, each $\Lambda_i^F$ is a connected order ideal of $P$, although $\Lambda_i^F$ is not necessarily a principal order ideal of $P$. Then by the definition of $G_P$, each $\Lambda_i^F$ is a vertex of $G_P$.
Moreover, we have the following result.

\begin{lem}\label{maxi}
For any $P$-forest $F\in\mathscr{F}(P)$,  the principal order ideals $\Lambda_1^F,\Lambda_2^F,\ldots,\Lambda_n^F$ form a maximum independent set of $G_P$.
\end{lem}

\begin{proof}
We first show that $\{\Lambda_1^F,\Lambda_2^F,\ldots,\Lambda_n^F\}$ is an independent set of $G_P$, that is, for any two nodes $i$, $j$ of $F$, the principal order ideals $\Lambda_i^{F}$ and $\Lambda_j^{F}$ are not adjacent in $G_P$.
There are  two cases to consider.

\begin{itemize}
\item[(1)]  The vertices $i$ and $j$ are incomparable in $F$. Since $F$ is a forest, it is clear that  $\Lambda_i^F\cap\Lambda_j^F=\emptyset$. This implies that
 $\Lambda_i^F$ and $\Lambda_j^F$ are not adjacent in $G_P$.

\item[(2)] The vertices  $i$ and $j$ are comparable in $F$. If  $i<_F j$, then  $\Lambda_i^F\subset\Lambda_j^F$; If $j<_F i$, then  $\Lambda_j^F\subset\Lambda_i^F$. In both circumstances, $\Lambda_i^F$ and $\Lambda_j^F$ are not adjacent in $G_P$.
\end{itemize}

We proceed  to show that the independent set $\{\Lambda_1^F,\Lambda_2^F,\ldots,\Lambda_n^F\}$ is of the largest possible size. To this end,  it is enough to verify that $| M|\leq n$ for any independent set $M$ of $G_P$.
Assume that $ M=\{J_1,J_2,\ldots,J_k\}$ is an independent set of $G_P$, which means that $J_i$ is a connected order ideal of $P$, and $J_i,J_j$ are not adjacent in $G_P$ for any $1\leq i< j\leq k$. We further  assume that the subscript satisfies $r<s$ whenever $J_r\subset J_s$.
In fact, this can be achieved as follows. Consider  $ M$ as a poset  ordered by set inclusion.
Then choose a subscript such that  $J_1J_2 \cdots J_k$ is  a linear extension of  $M$. Such a subscript satisfies the condition that $r<s$ whenever $J_r\subset J_s$.

For $1\leq s\leq k$, let
\[I_s=\bigcup_{1\leq r\leq s}J_r.\]
It is clear that $I_{s-1}\subseteq I_s$ for any $1<s\leq k$.
We  claim that
\begin{equation}\label{inclusion}
  \emptyset\neq I_1\subset I_2\subset\cdots\subset I_k\subseteq \{1,2,\ldots,n\},
\end{equation}
which implies  that $| M|=k\leq n$.

Suppose to the contrary  that $I_s= I_{s-1}$  for some $1< s\leq k$. Thus,
\begin{equation}\label{inclusion1}
  J_s\subseteq I_s= I_{s-1}=\bigcup_{1\leq r\leq s-1}J_{r}.
\end{equation}
The set $U(M,J_s)$ is defined as
\[U(M,J_s)=\{J'\mid J'\in M,J\subset J_s\}=\{J_r\mid 1\leq r\leq s-1, ~J_r\subset J_s\}.\]
Clearly,
\begin{equation}\label{bz1}
\mu(M,J_s)=\bigcup_{J'\in U(M,J_s)}J'\subseteq J_s.
\end{equation} Notice that for any $1\leq r\leq s-1$,
if $J_r$ does not belong to $U(M,J_s)$, then $J_r\cap J_s=\emptyset$, since otherwise $J_r$ and $J_s$ intersect nontrivially, contradicting the assumption that $M$ is an independent set of $G_P$.
 In view of relation \eqref{inclusion1}, we have
 \[J_s\subseteq \bigcup_{J'\in U(M,J_s)}J'=\mu(M,J_s),\]
 which together  with \eqref{eqn-umu} and \eqref{bz1}, leads to
 \[J_s=\mu(M,J_s)=\bigcup_{J'\in U_{max}(M,J_s)}J'.\]
If $U_{max}(M,J_s)$ has only one element, say, $U_{max}(M,J_s)=\{J_r\}$ for some $1\leq r\leq s-1$,  then $J_s=J_r$, which is contrary to $J_r\subset J_s$.
Next we may assume that $U_{max}(M,J_s)$ has more than one element.
By Lemma \ref{lemm-empty}, the intersection of any two elements of $U_{max}(M,J_s)$ is empty. Thus $J_s$ is the union of some (at least two) nonintersecting connected order ideals, which can not be connected. This contradicts the fact that $J_s$ is a connected order ideal. It follows that $I_{s-1}\subset I_s$ for each $1<s\leq k$, as desired.
\end{proof}

By the above lemma,  we can define a map $\Phi: \  \mathscr{F}(P) \longrightarrow  \mathscr{M}(G_P)$ by letting $$\Phi(F)=\{\Lambda_1^F,\Lambda_2^F,\ldots,\Lambda_n^F\}$$
for any $F\in \mathscr{F}(P)$.
In order to show that $\Phi$ is a bijection, we shall construct
the inverse map of $\Phi$, denoted by $\Psi$. To give a description of $\Psi$, we  need the following  lemma.

\begin{lem}\label{the independent set K}
Given $M\in \mathscr{M}(G_P)$ and $J\in M$, there exists a unique $j$ such that
\begin{equation}\label{single point2}
J\setminus\mu(M,J)=\{j\},
\end{equation}
where $\mu(M,J)$ is given in \eqref{def of M_J}.
Moreover, $j$ is a maximal element of $J$ with respect to the order $\leq_P$, and
\begin{equation}\label{different J}
  J_r\setminus \mu(M,J_r)\neq J_s\setminus\mu(M,J_s)
\end{equation}
for any distinct $J_r,J_s\in M$.
\end{lem}
\begin{proof}
By Lemma \ref{maxi}, we see that each maximum independent set of $G_P$ should contain $n$ vertices. Suppose that $M=\{J_1,J_2,\ldots,J_n\}$. As in the proof of Lemma \ref{maxi}, we may assume that
\begin{align}\label{eqn-cond}
r<s \mbox{ whenever } J_r\subset J_s.
\end{align}
For $1\leq s\leq n$, let
\[I_s=\bigcup_{1\leq r\leq s}J_r.\ \ \ \ \  \]
By \eqref{inclusion},  we see that
\begin{equation}\label{X}
\emptyset\neq I_1\subset I_2\subset\cdots\subset I_n\subseteq \{1,2,\ldots,n\}.
\end{equation}
Therefore, if setting  $I_0=\emptyset$, we obtain that for  $1\leq s\leq n$,
\begin{equation}\label{benb}
|I_s\setminus I_{s-1}|=1.
\end{equation}

Let $J=J_s$ for some $1\leq s\leq n$. In view of \eqref{def of M_J} and \eqref{eqn-cond}, we get that
$$\mu(M,J_s)= \bigcup_{J'\in M,~J'\subset J_s}J'=\bigcup_{1\leq r\leq s-1, J_r\subset J_s} J_r\subseteq I_{s-1}.$$
Thus we have
\begin{equation}\label{sequence properly include}
  J\setminus\mu(M,J)=J_s\setminus \mu(M,J_s)=J_s\setminus I_{s-1}=I_s\setminus I_{s-1},
\end{equation}
where the second equality follows from the fact that for any $1\leq r\leq s-1$, either $J_r\subset J_s$ or $J_r\cap J_s=\emptyset$.
In view of \eqref{benb} and \eqref{sequence properly include}, we arrive at \eqref{single point2} and \eqref{different J}.

It remains to show that the unique element $j$ of $J_s\setminus \mu(M,J_s)$ is a maximal element of $J_s$ with respect to the order $\leq_P$.
Suppose that $j$ is not a maximal element of $J_s$. Then there exists a maximal element $i$ of $J_s$ such that $j<_P i$. By \eqref{single point2} and $j\neq i$, we see that $i\in \mu(M,J_s)$. Therefore, there exists some $J'\subset J_s$ of and $J'\in M$
such that $i\in J'$. Since $J'$ is an order ideal of $P$, we get
$j\in J'\subseteq \mu(M,J_s)$, contradicting with the fact $j\not\in \mu(M,J_s)$.
\end{proof}

For any $M\in \mathscr{M}(G_P)$, it follows from \eqref{single point2} and \eqref{different J} that
\[\{1,2,\ldots,n\}=\biguplus_{J\in M} J\setminus\mu(M,J).\]
Let $F_M$ be the poset on $\{1,2,\ldots,n\}$  such that
 $i<_{F_M} j $  if and only if  $J_a\subset J_b$,
where $J_a$ and $J_b$ are the two connected order ideals in $M$ satisfies $J_a\setminus \mu(M,J_a)=\{i\}$, $J_b\setminus \mu(M,J_b)=\{j\}$.
The following result show an important property
for principal order ideals of the poset $F_M$.

\begin{lem}\label{the principal ideal of F}
Given $M\in \mathscr{M}(G_P)$, let $F_M$ be the poset defined as above.
Then for any $1\leq j\leq n$ we have $\Lambda_{j}^{F_M}=\{i\mid i\leq_{F_M}j\}=J$, where $J\in M$ satisfying $J\setminus\mu(M,J)=\{j\}$ as in Lemma \ref{the independent set K}.
\end{lem}

\noindent
\emph{Proof.}
We use the principle of Noetherian induction.

If $j$ is a minimal element of ${F_M}$ with respect to the order $\leq_{F_M}$, then $J$ is also a minimal element of $M$ when $M$ is regarded as a poset ordered by set inclusion. Hence $\Lambda_j^{{F_M}}=\{j\}$ and there exists no $J'\in M$ such that $J'\subset J$, which yields that $\mu(M,J)=\emptyset.$ So $J=\{j\}\cup \mu(M,J)=\{j\}$, and then $\Lambda_j^{F_M}=J$.

Suppose that $j$ is not a minimal element of ${F_M}$ (with respect to the order $\leq_{F_M}$) and $\Lambda_i^{{F_M}}=J'$ holds for any $i<_{{F_M}}j$, where  $J'\setminus \mu(M,J')=\{i\}$. The construction of ${F_M}$ tells us that   $i<_{{F_M}}j$ if and only if $J'\subset J$. Since $\Lambda_{i}^{{F_M}}\subset \Lambda_j^{{F_M}}$ holds for each $i<_{{F_M}}j$, we have
\begin{align*}
  \Lambda_j^{{F_M}} =\{i\mid i\leq_{{F_M}}j\}=\{j\}\cup\left( \bigcup_{i<_{{F_M}}j}\Lambda_{i}^{{F_M}}\right).
\end{align*}
Then by the induction hypothesis, we get that
\begin{align*}
  \Lambda_j^{{F_M}}=\{j\}\cup \left(\bigcup_{J'\in M,~J'\subset J}J'\right)= \{j\}\cup \mu(M,J)= J.  \tag*{\qed}
\end{align*}

We proceed to examine more structure of $F_M$, and obtain the following result.

\begin{lem}\label{M is a forest}
For any $M\in \mathscr{M}(G_P)$, the poset $F_M$ is a $P$-forest.
\end{lem}

\begin{proof}
We first show that ${F_M}$ is   a forest.
Suppose otherwise that  ${F_M}$ is not a forest. Then there exists an element $i$ in ${F_M}$ such that $i$ is covered by at least two elements of ${F_M}$, say $j,k$. Thus $j$ and $k$ must be incomparable with respect to the order $\leq_{F_M}$. (Recall that in a poset $P$, we say that an element
$u$ is  covered by an element $v$ if $u<_P v$ and there is no element $w$ such that $u<_P w<_P v$.)
By Lemma \ref{the independent set K}, there exist $J_a, J_b,J_c\in M$ such that $J_a\setminus \mu(M,J_a)=\{i\}$, $J_b\setminus \mu(M,J_b)=\{j\}$ and $J_c\setminus \mu(M,J_c)=\{k\}$.
By the construction of  ${F_M}$, we see that  $J_a\subset J_b$, $J_a\subset J_c$  and $J_b$, $J_c$ are incomparable in $M$ with respect to the set inclusion order.  Hence, $J_b\not\subset J_c$, $J_c\not\subset J_b$  and
 $ (J_b\cap J_c)\supseteq J_a\neq \emptyset$. This implies that $J_b$ and $J_c$ are adjacent in the graph   $G_P$, contradicting  the fact that $  M$ is an independent set.

We proceed to show that ${F_M}$ is a $P$-forest. By Lemma \ref{the principal ideal of F}, for each element $i$ of $F_M$, the subtree
 $\Lambda_i^{{F_M}}=\{j\mid j\leq_{{F_M}}i\}$ of ${F_M}$
 rooted at $i$ is a connected order ideal of $P$.
To verify that ${F_M}$ is a $P$-forest, we still need to check that for $1\leq i,j\leq n$, if  $i$ and $j$  are incomparable in ${F_M}$, then the union $\Lambda_{i}^{{F_M}}\cup\Lambda_{j}^{{F_M}}$ is a disconnected order ideal of $P$.
By Lemma \ref{the independent set K}, assume that $J_a$ and $J_b$ are the connected order ideals in $M$ such that $J_a\setminus \mu(M,J_a)=\{i\}$ and $J_b\setminus \mu(M,J_b)=\{j\}$.
By Lemma \ref{the principal ideal of F},  we have $J_a=\Lambda_{i}^{{F_M}}$ and $J_b=\Lambda_{j}^{{F_M}}$. Since  $i$ and $j$  are incomparable in ${F_M}$, we obtain that  $J_a\not\subset J_b$ and $J_b\not\subset J_a$. On the other hand, $J_a$ and $J_b$ are not adjacent in the graph $G_P$. This allows us to conclude that  $J_a\cap J_b=\emptyset$. Therefore, as an order ideal of $P$, the union $J_a\cup J_b$ is disconnected, so is the union $\Lambda_{i}^{{F_M}}\cup\Lambda_{j}^{{F_M}}$. Hence ${F_M}$ is a $P$-forest.
\end{proof}

With the above lemma, we can define the inverse map of $\Phi$, denoted by $\Psi: \mathscr{M}(G_P)\rightarrow \mathscr{F}(P)$, by letting
$$\Psi(M)=F_M$$
for any $M\in \mathscr{M}(G_P)$.

Now we  are in a position to give a proof of Theorem \ref{The First Main Result}.

\noindent \textit{Proof of Theorem \ref{The First Main Result}.}
We first prove that $\Psi(\Phi (F))=F$ for any $P$-forest $F$ and 
$\Phi(\Psi (M))=M$ for any maximum independent set $M$ of $G_P$. 
The proof of the former statement will be given below, and the proof of the latter will be omitted here.
Given a $P$-forest $F$, by definition, the image of $F$ under the map $\Phi$ is  $\Phi(F)=\{\Lambda_1^F, \ldots, \Lambda_n^F\}$, which is a maximum independent set of $G_P$ by Lemma \ref{maxi}. Of course, we have $\Lambda_i^F\subset\Lambda_j^F$ if and only if $i<_F j$.  For each $1\leq i\leq n$ let $J_i=\Lambda_i^F$ and then denote $M=\{J_1,J_2,\ldots,J_n\}$. We proceed to show that $\Psi(M)=F_M=F$. Note that both $F_M$ and $F$ are posets on $\{1,2,\ldots,n\}$. It remains to show that
$i<_{F_m} j$ if and only if $i<_{F} j$ for any $i,j\in \{1,2,\ldots,n\}$.
Recall that for $1\leq i\leq n$ the principal order ideal $\Lambda_i^F$ is the subtree of $F$ rooted at $i$. Hence 
\[J_i\setminus \mu(M,J_i)=\Lambda_i^F\setminus \left( \bigcup_{j<_F i}\Lambda_j^F  \right)=\{i\}\]
holds for each $1\leq i\leq n$.
By the construction of ${F_M}$, we know that $i<_{F_M}j$ if and only if $J_i\subset J_j$. On the other hand, in the given $P$-forest $F$, $i<_F j$ if and only if $\Lambda_i^F\subset\Lambda_j^F$. Since $J_i=\Lambda_i^F$ for each $1\leq i\leq n$, it follows that $i<_{F_M}j$ if and only if $i<_F j$. Thus $F_M=F$, as desired.

The quality $\Phi(\Psi (M))=M$ ensures that $\Phi$ is onto. Hence $\Phi$ is bijective.
\qed

We take the poset $P$ in Figure \ref{a typcial example} as an example to illustrate Theorem \ref{The First Main Result} and its proof. There are there $P$-forests $F_1, F_2$ and $F_3$ as shown in Figure \ref{a typcial example}.
The graph $G_P$, as shown in Figure \ref{posets}, has three maximum independent sets:
\begin{eqnarray*}
  M^1&=&\{\Lambda_3^P,\Lambda_4^P,\Lambda_6^P,\Lambda_1^P,\Lambda_2^P,\Lambda_{2,5}^P\}, \\[5pt]
 M^2&=&\{\Lambda_3^P,\Lambda_4^P,\Lambda_6^P,\Lambda_1^P,\Lambda_{1,5}^P,\Lambda_{2,5}^P\}, \\[5pt]
  M^3&=&\{\Lambda_3^P,\Lambda_4^P,\Lambda_6^P,\Lambda_5^P,\Lambda_{1,5}^P,\Lambda_{2,5}^P\}.
\end{eqnarray*}
The principal order ideals of $F_1$ is as shown in Figure \ref{The P-forest F}.

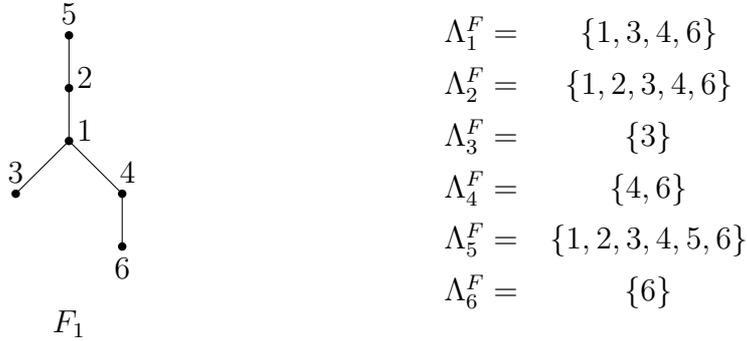
\begin{figure}[h]
\begin{minipage}{0.45\textwidth}
\begin{center}
    \begin{tikzpicture}[scale=0.7]
\draw (0,0)--(1,1)--(2,0)--(2,-1) (1,1)--(1,2)--(1,3);
\fill (1,3) circle(0.08cm); \coordinate[label=above:$5$] (1) at (1,3);
\fill (1,2) circle(0.08cm); \coordinate[label=above:$2$] (1) at (1.3,1.8);
\fill (0,0) circle(0.08cm); \coordinate[label=above:$3$] (1) at (0,0);
\fill (1,1) circle(0.08cm); \coordinate[label=above:$1$] (1) at (1.3,0.8);
\fill (2,0) circle(0.08cm); \coordinate[label=above:$4$] (1) at (2.1,0);
\fill (2,-1) circle(0.08cm); \coordinate[label=below:$6$] (1) at (2,-1);
\coordinate[label=below:$F_1$] (4) at (1,-2);
\end{tikzpicture}
\end{center}
\end{minipage}
\begin{minipage}{.45\textwidth}
\begin{center}
\begin{tabular}{cc}
$\Lambda_1^F=$&$\{1,3,4,6\}$ \\ [5pt]
$\Lambda_2^F=$&$\{1,2,3,4,6\}$ \\ [5pt]
$\Lambda_3^F=$&$\{3\}$ \\ [5pt]
$\Lambda_4^F=$&$\{4,6\}$ \\ [5pt]
$\Lambda_5^F=$&$\{1,2,3,4,5,6\}$  \\ [5pt]
$\Lambda_6^F=$&$\{6\}$  \\ [5pt]
\end{tabular}
\end{center}
\end{minipage}
\caption{The $P$-forest $F_1$ and its principal order ideals.}
\label{The P-forest F}
\end{figure}
\noindent
By the construction of $\Phi$, we have
\begin{eqnarray*}
 \Phi(F_1) &=&\{\Lambda_1^{F_1},\Lambda_2^{F_1},\ldots,\Lambda_6^{F_1}\}  \\
   &=& \big\{\{1,3,4,6\},\{1,2,3,4,6\},\{3\},\{4,6\},\{1,2,3,4,5,6\},\{6\}\big\},
\end{eqnarray*}
which coincides with $M^1$.
One can also  verify  that $\Phi(F_2)=M^2$ and $\Phi(F_3)=M^3$.

On the other hand, for the maximum independent set $M^1$, if we set $J_1=\Lambda_1^P=\{1,3,4,6\},$ $J_2=\Lambda_2^P=\{1,2,3,4,6\},$ $J_3=\Lambda_3^P=\{3\},$ $J_4=\Lambda_4^P=\{4,6\},$ $J_5=\Lambda_{2,5}^P=\{1,2,3,4,5,6\},$ $J_6=\Lambda_6^P=\{6\}$, then it is straightforward to verify that $J_i\setminus \mu(M^1,J_i)=\{i\}$ for $1\leq i\leq 6$. And then, by definition, in the $P$-forest $F_{M^1}$ there is $2<_{F_{M^1}}5$, $1<_{F_{M^1}}2$, $3<_{F_{M^1}}1$, $4<_{F_{M^1}}1$, $6<_{F_{M^1}}4$. One  readily sees that $F_{M^1}=F_1$. Similarly, one can verify that $F_{M^2}=F_2$ and $F_{M^3}=F_3$.

\section{$F_P(\textbf{x})$ for naturally labeled $P$} \label{sect-main}

The main objective of this section is to prove Theorems \ref{def-descent} and \ref{the seconde main result}. The proofs are based on some properties of certain subgraphs of $G_P$. Although we require that the poset $P$ in Theorems \ref{def-descent} and \ref{the seconde main result} be naturally labeled, these properties of $G_P$ are valid for any finite poset $P$.

To begin with, let us first introduce some notations. For an order ideal $J$ of $P$,  let $gs(J)$ denote    the set of maximal elements of $J$ with respect to the order $\leq_P$, namely,
\[gs(J)=\{i\in J\mid \text{ there exists no } j\in J \text{ such that } i<_P j\}.\]
This set is also called the generating set of $J$.
Clearly, when $gs(J)=\{i_1,i_2,\ldots,i_k\}$, we have
$J=\Lambda_{i_1}^P\cup\Lambda_{i_2}^P\cup\cdots\cup\Lambda_{i_k}^P.$
Let $\chi_J$  be the subgraph of $G_P$ induced by the vertex subset $\{\Lambda_{i_1}^P,\Lambda_{i_2}^P,\ldots,\Lambda_{i_k}^P\}$.
We have the following  assertion.

\begin{lem}\label{induced graph}
For any connected order ideal $J$ of $P$, the graph $\chi_J$ is connected.
\end{lem}

\begin{proof} Assume that $gs(J)=\{i_1,i_2,\ldots,i_k\}$. The proof is immediate if $k=1$. In the following we shall assume that $k\geq 2$.
Define
\[
  \mathrm{Conn}(i_1)=\big\{i_r\in gs(J)\,|\, \text{there is a path in $\chi_J$ connecting $\Lambda_{i_1}^P$ and $\Lambda_{i_r}^P$}  \}.
\]
Note that $i_1$ is always contained in $\mathrm{Conn}(i_1)$.
It is enough  to show that $\mathrm{Conn}(i_1)=gs(J)$.
Otherwise, suppose that $\mathrm{Conn}(i_1)\neq gs(J)$.
Let
\[I_1=\bigcup_{j\in \mathrm{Conn}(i_1)}\Lambda_{ j}^P\ \ \ \ \text{and}\ \ \ \ I_2=\bigcup_{j\in gs(J)\setminus \mathrm{Conn}(i_1)}\Lambda_{ j}^P.\]
Then both $I_1$ and $I_2$ are nonempty subsets of $J$ satisfying that $I_1\cup I_2=J$, and both $I_1$ and $I_2$ are order ideals of $P$. Since  $J$ is a connected  order ideal of $P$, it follows that  $I_1\cap I_2\neq \emptyset$. Thus there exists some $u\in \mathrm{Conn}(i_1)$ and some $v\in gs(J)\setminus \mathrm{Conn}(i_1)$  such that $\Lambda_{u}^P  \cap \Lambda_{v}^P  \neq \emptyset$. Since both $u$ and $v$ are maximal elements in the connected order ideal $J$, we must have $\Lambda_{u}^P\not\subset\Lambda_{v}^P$
and $\Lambda_{v}^P\not\subset\Lambda_{u}^P$. This means  that $\Lambda_{u}^P$ and $\Lambda_{v}^P$ are adjacent, implying that $v\in \mathrm{Conn}(i_1)$. This leads to a contradiction.
\end{proof}

We also need the following lemma.

\begin{lem}\label{absorbtion}
Let $J$ be  a connected order ideal of $P$, and let $C$ be any connected subgraph of $G_P$. Assume that $J$ is not adjacent to any vertex of $C$. If there exists a vertex $J_a$ of $C$ such that $J_a\subset J $, then $J_b\subset J$ for any vertex $J_b$ of $C$.
\end{lem}

\begin{proof}
We first consider the case when $J_a$ and $J_b$ are adjacent. In this case,  $J_b$ and $J_a$ intersect nontrivially, and so we have  $\emptyset\neq (J_a\cap J_b)$. On the other hand, since $J_a\subset J $, we obtain that
\begin{equation}\label{xc}
\emptyset\neq (J_a\cap J_b)\subset (J\cap J_b).
\end{equation}
Combining \eqref{xc} and the hypothesis that the vertices $J_b$ and $J$ are not adjacent, we get that $J_b\subset J$ or $J\subset J_b$. If $J\subset J_b$, then $J_a\subset J\subset J_b$,
which is impossible because $J_a$ and $J_b$ intersect nontrivially. Hence we have $J_b\subset J$.

We now consider the case when $J_a$ is not adjacent to $J_b$. Since $C$ is connected, there exists a sequence $(J_0=J_a, J_1,\ldots, J_k=J_b)$ ($k\geq 2$) of vertices of $C$  such that $J_i$ is adjacent to $J_{i-1}$ for $1\leq i\leq k$.
By the above argument, $J_1$ is contained in $J$. Therefore,  by a simple recursion we get that $J_b\subset J$.
\end{proof}

For example, let $P$ be
  the poset  given in Figure \ref{The poset P}. The graph  $G_P$ is illustrated in Figure \ref{The graph G_P}, where we adopt the notation $\Lambda_{i,j}^P=\Lambda_i^P\cup\Lambda_j^P$ and $\Lambda_{i,j,k}^P=\Lambda_i^P\cup\Lambda_j^P\cup\Lambda_k^P$. The graph $G_P$ has totally 13 connected components, and among them there are four connected components $C_1,$ $C_2,$  $C_3,$  $C_4$ which have more than one vertex.
 \begin{itemize}
   \item To illustrate the assertion of Lemma \ref{induced graph}, for example,
let $J=\Lambda_{4,5,6}^P$, then we have $gs(J)=\{4,5,6\}$. One can verify that  the subgraph $\chi_J$ of $G_P$ induced by the vertex subset $\{\Lambda_{4}^P,$ $\Lambda_{5}^P,$ $\Lambda_{6}^P\}$ is indeed connected.

   \item To illustrate the assertion of Lemma \ref{absorbtion}, for example, we let $J=\Lambda_{10}^P$, and let $C$ be the connected component $C_1$ of $G_P$, then $\Lambda_5^P\subset J$.  In this case we see that $J'\subset \Lambda_{10}^P$ for any $J'\in V(C_1)$.
 \end{itemize}

 \begin{figure}[h]
\begin{center}
\begin{tikzpicture}[scale=0.9]
\draw (0,1)--(0,2)--(1,1)--(1,2)--(2,1)--(2,2)--(0,1);

\draw (2.5,4)--(2.5,5.5)--(4,4)--(4,5.5);

\draw (0,2)--(2.5,4)--(1,2)--(4,4)--(2,2)--(2.5,4) (0,2)--(4,4);

\draw (2.5,5.5)--(2.5,7)--(4,5.5)--(4,7)--(2.5,5.5);

\fill (0,1) circle(0.05cm); \coordinate[label=below:$1$] (1) at (0,1);
\fill (1,1) circle(0.05cm); \coordinate[label=below:$2$] (1) at (1,1);
\fill (2,1) circle(0.05cm); \coordinate[label=below:$3$] (1) at (2,1);
\fill (0,2) circle(0.05cm); \coordinate[label=below:$4$] (1) at (-0.1,2);
\fill (1,2) circle(0.05cm); \coordinate[label=below:$5$] (1) at (0.9,2);
\fill (2,2) circle(0.05cm); \coordinate[label=below:$6$] (1) at (1.9,2);

\draw (5,2.5)--(5,1.5)--(6,0.5)--(6,2.5)--(7,1.5);
\draw(2.5,4)--(5,2.5)--(4,5.5) (2.5,4)--(6,2.5)--(4,5.5);
\fill (7,1.5) circle(0.05cm); \coordinate[label=below:$8$] (1) at (7,1.5);
\fill (6,2.5) circle(0.05cm); \coordinate[label=below:$11$] (1) at (5.7,2.5);
\fill (6,0.5) circle(0.05cm); \coordinate[label=below:$7$] (1) at (6,0.5);
\fill (5,1.5) circle(0.05cm); \coordinate[label=below:$9$] (1) at (4.7,1.5);
\fill (5,2.5) circle(0.05cm); \coordinate[label=below:$12$] (1) at (4.7,2.5);
\fill (2.5,4) circle(0.05cm); \coordinate[label=above:$13$] (1) at (2.2,4);
\fill (4,4) circle(0.05cm); \coordinate[label=above:$10$] (1) at (4.2,3.3);
\fill (2.5,5.5) circle(0.05cm); \coordinate[label=above:$14$] (1) at (2.2,5.4);

\fill (4,5.5) circle(0.05cm); \coordinate[label=above:$15$] (1) at (4.3,5.3);

\fill (2.5,7) circle(0.05cm); \coordinate[label=above:$16$] (1) at (2.5,7);
\fill (4,7) circle(0.05cm); \coordinate[label=above:$17$] (1) at (4,7);
\end{tikzpicture}
\end{center}
\caption{A naturally labeled poset $P$.}
\label{The poset P}
\end{figure}
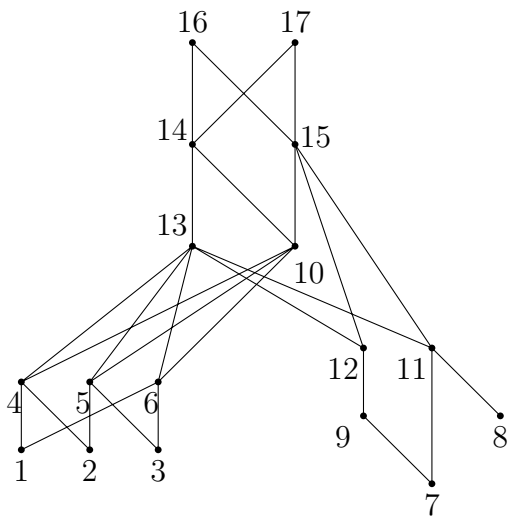

\begin{figure}[h]
\begin{center}
\begin{tikzpicture}[scale=1]
\draw (0,4)--(1,5)--(2,4)--(2,1)--(1,2)--(0,1)--(0,4) (1,2)--(1,5) (2,1)--(2,4) (0,1)--(2,1)  (0,4)--(2,4);
\draw (5,2)--(5,3)--(5,4);
\draw (4,1)--(4,2)--(4,3)--(4,4)--(4,5) (4,3)--(3,2);

\draw (5,1)--(5,2) (6,1)--(6,2);

\fill (0,4) circle(0.1cm); \coordinate[label=above:$\Lambda_{4,6}^P$] (1) at (-0.1,4);
\fill (1,5) circle(0.1cm); \coordinate[label=above:$\Lambda_{4,5}^P$] (1) at (1,5);
\fill (2,4) circle(0.1cm); \coordinate[label=above:$\Lambda_{5,6}^P$] (1) at (2,4);

\fill (0,1) circle(0.1cm); \coordinate[label=below:$\Lambda_{5}^P$] (1) at (0,1);
\fill (1,2) circle(0.1cm); \coordinate[label=below:$\Lambda_{6}^P$] (1) at (1,1.8);
\fill (2,1) circle(0.1cm); \coordinate[label=below:$\Lambda_{4}^P$] (1) at (2,1);

\fill (4,1) circle(0.1cm); \coordinate[label=below:$\Lambda_{13,15}^P$] (1) at (3.7,1);
\fill (4,2) circle(0.1cm); \coordinate[label=above:$\Lambda_{14}^P$] (1) at (4.3,2);
\fill (4,3) circle(0.1cm); \coordinate[label=above:$\Lambda_{15}^P$] (1) at (3.7,3);
\fill (4,4) circle(0.1cm); \coordinate[label=above:$\Lambda_{13}^P$] (1) at (3.7,4);
\fill (4,5) circle(0.1cm); \coordinate[label=above:$\Lambda_{10}^P$] (1) at (3.7,5);
\fill (3,2) circle(0.1cm); \coordinate[label=below:$\Lambda_{10,13}^P$] (1) at (3,2);

\fill (5,1) circle(0.1cm); \coordinate[label=below:$\Lambda_{9}^P$] (1) at (5,1);
\fill (6,1) circle(0.1cm); \coordinate[label=below:$\Lambda_{17}^P$] (1) at (6,1);
\fill (5,2) circle(0.1cm); \coordinate[label=above:$\Lambda_{11}^P$] (1) at (5.3,2);
\fill (5,3) circle(0.1cm); \coordinate[label=above:$\Lambda_{12}^P$] (1) at (5.3,3);

\fill (6,2) circle(0.1cm); \coordinate[label=above:$\Lambda_{16}^P$] (1) at (6.3,2);

\fill (5,5) circle(0.1cm); \coordinate[label=above:$\Lambda_{1}^P$] (1) at (5,5);
\fill (6,5) circle(0.1cm); \coordinate[label=above:$\Lambda_{2}^P$] (1) at (6,5);
\fill (7,5) circle(0.1cm); \coordinate[label=above:$\Lambda_{3}^P$] (1) at (7,5);
\fill (5,4) circle(0.1cm); \coordinate[label=above:$\Lambda_{9,11}^P$] (1) at (5,4);
\fill (6,4) circle(0.1cm); \coordinate[label=above:$\Lambda_{8}^P$] (1) at (6,4);
\fill (7,4) circle(0.1cm); \coordinate[label=above:$\Lambda_{7}^P$] (1) at (7,4);

\fill (8,5) circle(0.1cm); \coordinate[label=above:$\Lambda_{4,5,6}^P$] (1) at (8,5);
\fill (8,4) circle(0.1cm); \coordinate[label=above:$\Lambda_{14,15}^P$] (1) at (8,4);
\fill (8,3) circle(0.1cm); \coordinate[label=above:$\Lambda_{11,12}^P$] (1) at (8,3);
\fill (8,2) circle(0.1cm); \coordinate[label=above:$\Lambda_{16,17}^P$] (1) at (8,2);
\coordinate[label=below:$C_1$] (1) at (1,0.3);
\coordinate[label=below:$C_2$] (1) at (3.7,0.3);
\coordinate[label=below:$C_3$] (1) at (5,0.3);
\coordinate[label=below:$C_4$] (1) at (6,0.3);
\end{tikzpicture}
\end{center}
\caption{The graph $G_P$ associated to the poset $P$ in Figure \ref{The poset P}.}
\label{The graph G_P}
\end{figure}
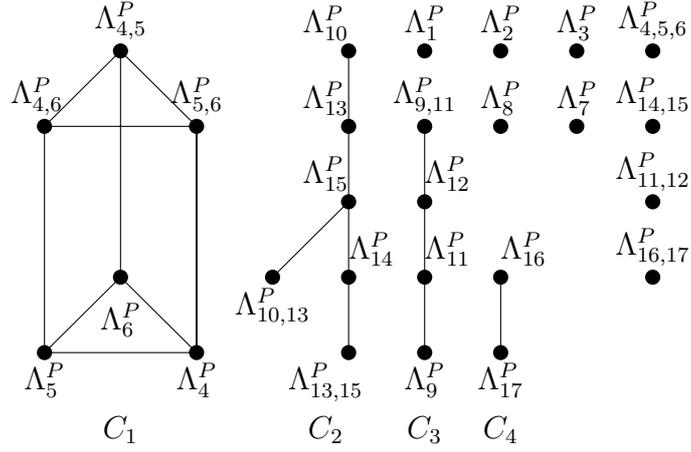

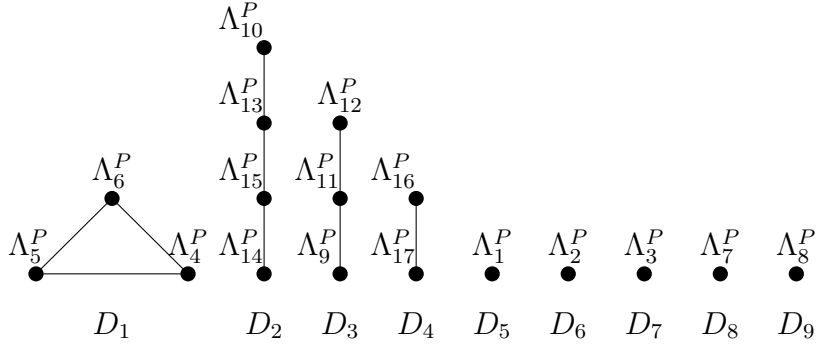
\begin{figure}[h]
\begin{center}
\begin{tikzpicture}[scale=1]
\draw (0,3)--(1,4)--(2,3)--(0,3);

\draw (3,3)--(3,4)--(3,5)--(3,6);

\draw (4,3)--(4,4) (5,3)--(5,4);

\fill (0,3) circle(0.1cm); \coordinate[label=above:$\Lambda_5^P$] (1) at (-0.1,3);
\fill (1,4) circle(0.1cm); \coordinate[label=above:$\Lambda_6^P$] (1) at (1,4);
\fill (2,3) circle(0.1cm); \coordinate[label=above:$\Lambda_4^P$] (1) at (2,3);

\fill (3,3) circle(0.1cm); \coordinate[label=above:$\Lambda_{14}^P$] (1) at (2.7,3);
\fill (3,4) circle(0.1cm); \coordinate[label=above:$\Lambda_{15}^P$] (1) at (2.7,4);
\fill (3,5) circle(0.1cm); \coordinate[label=above:$\Lambda_{13}^P$] (1) at (2.7,5);
\fill (3,6) circle(0.1cm); \coordinate[label=above:$\Lambda_{10}^P$] (1) at (2.7,6);

\fill (4,3) circle(0.1cm); \coordinate[label=above:$\Lambda_{9}^P$] (1) at (3.7,3);
\fill (5,3) circle(0.1cm); \coordinate[label=above:$\Lambda_{17}^P$] (1) at (4.7,3);
\draw (4,5)--(4,4);
\fill (4,4) circle(0.1cm); \coordinate[label=above:$\Lambda_{11}^P$] (1) at (3.7,4);
\fill (5,4) circle(0.1cm); \coordinate[label=above:$\Lambda_{16}^P$] (1) at (4.7,4);

\fill (6,3) circle(0.1cm); \coordinate[label=above:$\Lambda_{1}^P$] (1) at (6,3);
\fill (7,3) circle(0.1cm); \coordinate[label=above:$\Lambda_{2}^P$] (1) at  (7,3);
\fill (8,3) circle(0.1cm); \coordinate[label=above:$\Lambda_{3}^P$] (1) at (8,3);
\fill (4,5) circle(0.1cm); \coordinate[label=above:$\Lambda_{12}^P$] (1) at (4,5);
\fill (9,3) circle(0.1cm); \coordinate[label=above:$\Lambda_{7}^P$] (1) at (9,3);
\fill (10,3) circle(0.1cm); \coordinate[label=above:$\Lambda_{8}^P$] (1) at (10,3);

\coordinate[label=above:$D_{1}$] (1) at (1,2);
\coordinate[label=above:$D_{2}$] (1) at (3,2);
\coordinate[label=above:$D_{3}$] (1) at (4,2);
\coordinate[label=above:$D_{4}$] (1) at (5,2);
\coordinate[label=above:$D_{5}$] (1) at (6,2);
\coordinate[label=above:$D_{6}$] (1) at (7,2);
\coordinate[label=above:$D_{7}$] (1) at (8,2);
\coordinate[label=above:$D_{8}$] (1) at (9,2);
\coordinate[label=above:$D_{9}$] (1) at (10,2);
\end{tikzpicture}
\end{center}
\caption{The subgraph $H_P$ induced on $G_P$ by principal order ideals.}
\label{The graph H_P}
\end{figure}

Now we turn to study a special subgraph of $G_P$, which is induced by the principal order ideals of $P$. This graph also plays an important role in our future proofs.
Recall that the set of principal order ideals of $P$ consists of
$\Lambda_{1}^P,\Lambda_{2}^P,\ldots,\Lambda_{n}^P$. Let $H_P$ be the subgraph of $G_P$ induced by the vertex subset  $\{\Lambda_{1}^P,\Lambda_{2}^P,\ldots,\Lambda_{n}^P\}$. For example, for the poset $P$ and the graph $G_P$ as illustrated in Figures \ref{The poset P} and \ref{The graph G_P},  the graph $H_P$ is as shown in Figure \ref{The graph H_P}.
It follows from Lemma \ref{induced graph} that for a given connected order ideal $J$ the induced subgraph $\chi_J$ must be a subgraph of certain connected component of $H_P$, where $\chi_J$ is defined as before Lemma \ref{induced graph}. The graph $H_P$ admits the following interesting properties.

\begin{lem}\label{intersect trivially}
Suppose that $H_P$ has connected components $D_1, D_2,\ldots, D_\ell$.
We have the following two assertions.
\begin{description}
  \item[(1)] Let $1\leq r< s\leq \ell$, and let $J_a,J_b$ be two connected order ideals of $P$. If $\chi_{J_a}$ is a subgraph of $D_r$ while $\chi_{J_b}$ is a subgraph of $D_s$, then $J_a$ and $J_b$ are not adjacent in $G_P$.

  \item[(2)] Given a connected order ideal $J$, suppose that $\chi_J$ is a subgraph of the connected component $D_r$ of $H_P$, and hence $J\subseteq \bigcup_{\Lambda_i^P\in V(D_r)}\Lambda_i^P$. If $J\neq \bigcup_{\Lambda_i^P\in V(D_r)}\Lambda_i^P$, then there exists some $\Lambda_j^P\in V(D_r)$ such that $J$ and $\Lambda_j^P$  are adjacent in $G_P$.
\end{description}
\end{lem}
\begin{proof}
 Let us first prove assertion (1). Suppose to the contrary that  $J_a$ and $J_b$ are adjacent in the graph $G_P$. Then $J_a\cap J_b\neq \emptyset$. Since
 $$J_a=\bigcup_{i\in gs(J_a)}\Lambda_i^P, \quad J_b=\bigcup_{j\in gs({J_b})}\Lambda_j^P,$$
  there exist some $i\in gs(J_a)$ and $j\in gs({J_b})$ such that $\Lambda_i^P\cap \Lambda_j^P\neq \emptyset$. Notice that $\Lambda_i^P$ is a vertex of the connected component $D_r$ and $\Lambda_j^P$ is a vertex of the connected component $D_s$, so $\Lambda_i^P$ and $\Lambda_j^P$ are not adjacent in the graph $H_P$. Since the graph $H_P$ is a vertex induced subgraph of $G_P$,  the order ideals $\Lambda_i^P$ and $\Lambda_j^P$ are also not adjacent in the graph $G_P$, hence they intersect trivially. Because $\Lambda_i^P\cap \Lambda_j^P\neq \emptyset$, we must  have $\Lambda_i^P\subset\Lambda_j^P$ or $\Lambda_j^P\subset\Lambda_i^P$. If $\Lambda_i^P\subset\Lambda_j^P$,  by  Lemmas \ref{induced graph} and  \ref{absorbtion} we obtain  that  for any $k\in gs({J_a})$, there is $\Lambda_k^P\subset \Lambda_j^P$.  Then,
\[J_a=\bigcup_{k\in gs({J_a})}\Lambda_k^P\subset \Lambda_j^P\subseteq J_b,\]
which implies  that $J_a$ and $J_b$ are not  adjacent in the graph $G_P$.
If $\Lambda_j^P\subset\Lambda_i^P$, we can use a similar argument to deduce that $J_a$ and $J_b$  are not  adjacent in the graph $G_P$. In both cases, we are led to a contradiction.

We proceed to prove assertion (2). Recall that $V(D_r)$ denotes the set of vertices of $D_r$. Assume that $gs(J)=\{i_1,\ldots,i_k\}$. Since $J\subseteq \bigcup_{\Lambda_i^P\in V(D_r)}\Lambda_i^P$ but $J\neq \bigcup_{\Lambda_i^P\in V(D_r)}\Lambda_i^P$,  there exists some  $\Lambda_j^P\in V(D_r)$ such that $\Lambda_j^P\nsubseteq  J$. Let
\begin{eqnarray*}
  V_1&=& \{\Lambda_i^P\in V(D_r)\mid \Lambda_i^P\subseteq J \}, \\
  V_2&=& \{\Lambda_j^P\in V(D_r)\mid \Lambda_j^P\nsubseteq  J \}.
\end{eqnarray*}
Clearly, we have $V_1\cup V_2=V(D_r)$ and  $V_2\neq \emptyset$. Since $\chi_J$ is a subgraph of $D_r$, we see that $V_1\neq \emptyset$. Because $D_r$ is a connected component of $H_P$, there exist some $\Lambda_i^P\in V_1$ and $\Lambda_j^P\in V_2$ such that $\Lambda_i^P$ and $\Lambda_j^P$ are adjacent in the graph $H_P$. Since $H_P$ is a vertex induced subgraph of $G_P$, the vertices $\Lambda_i^P$ and $\Lambda_j^P$ are also adjacent in $G_P$, which means that $\Lambda_i^P$ and $\Lambda_j^P$ intersect nontrivially, namely
  \[\Lambda_i^P \cap \Lambda_j^P\neq\emptyset,~~~\Lambda_i^P\not\subset \Lambda_j^P, \text{ and }\Lambda_j^P\not\subset \Lambda_i^P.\]
 In view of that $\Lambda_i^P\subseteq J$ and $\Lambda_j^P\in V_2$,  we get $J\neq \Lambda_j^P$ and
\[J \cap \Lambda_j^P\neq\emptyset,~~~J\not\subset \Lambda_j^P, \text{ and }\Lambda_j^P\not\subset J.\]
Hence $J$ is adjacent to $\Lambda_j^P$, as desired.
\end{proof}

With the above lemma, we can further obtain another property of $G_P$.

\begin{lem}\label{isolated vertex}
Let $C_r$ be a connected component of $G_P$ with vertex set $V(C_r)$.
Let $J$ be a connected order ideal with the graph $\chi_J$ as defined as above.
We have the following two assertions:
\begin{description}
  \item[(1)] Let $J^{max}_r$ denote the set $\bigcup_{J'\in V(C_r)}J'$. Then $J^{max}_r$ is an isolated vertex of the graph $G_P$.

   \item[(2)] If $\chi_J$ is a subgraph of $C_r$, and $J\neq J^{max}_r$, then $J$ is a vertex of $C_r$.
\end{description}
\end{lem}

\begin{proof}
Let us first prove assertion (1). It is clearly true when $|V(C_r)|=1$. Suppose $|V(C_r)|\geq 2$. We first prove that $J_r^{max}$ is a connected order ideal. Let $V$ be a set of connected order ideals and  assume $V$ satisfies the condition:
\begin{align}
 \mbox{$V\subseteq V(C_r)$ and $\bigcup_{J\in V}J$ is a connected order ideal.} \tag{*}
\end{align}
We claim that if $V$ satisfies (*) and is of the largest possible size, then $V$ must be equal to $V(C_r)$. Otherwise, suppose $V\subset V(C_r)$ but $V\neq V(C_r)$. Since $C_r$ is a connected graph and $|V(C_r)|\geq 2$, there exist some $J_a\in V$ and $J_b\in (V(C_r)\setminus V)$ such that $J_a$ and $J_b$ are adjacent in $G_P$. Hence $J_a\cap J_b\neq \emptyset$, and then $(\bigcup_{J\in V}J)\cap J_b\neq \emptyset$. It follows that the set $V'=V\cup \{J_b\}$ also satisfies the condition (*), and $|V'|=|V|+1$, contradicting the assumption that  $V$ is of the largest possible size.

We mow prove  that $J^{max}_r$ is not adjacent to any other vertex of $G_P$. For a $J\in \mathcal{J}_{conn}(P)$, if $J\in V(C_r)$, then $J \subset J^{max}_r$ and so $J$ and $J^{max}_r$ are not adjacent in $G_P$. If $J\notin V(C_r)$, namely, $J$ is not adjacent to any vertex of $C_r$,
we need to consider three cases:
\begin{itemize}
  \item[(i)] There exists some $J_a\in V(C)$ such that $J_a\subset J$. Then by Lemma \ref{absorbtion} we obtain that $J_b \subset J$ for any other $J_b\in V(C_r)$. Hence $J^{max}_r\subset J$, and it follows that $J$ and $J^{max}_r$ are not adjacent in $G_P$;
  \item[(ii)]   There exists some $J_a\in V(C)$ such that $ J\subset  J_a$. Then $J\subset J^{max}_r$, and as a consequence, $J$ and $J^{max}_r$ are also not adjacent in $G_P$;
  \item[(iii)] $J\cap J_a=\emptyset$ for any $J_a\in V(C_r)$. Then $J^{max}_r\cap J=\emptyset$ and, again, $\widetilde{J}$ and $J$ are  not adjacent in $G_P$.
\end{itemize}
Hence we conclude that $J^{max}_r$ is an isolated vertex of the graph $G_P$.

To prove assertion (2), we first analyse some general properties of $G_P$. Suppose the graph $H_P$ has $\ell$ connected components $D_1,D_2,\ldots,D_\ell$. Lemma \ref{induced graph} tells us that for any  connected order ideal $J'$, the graph $\chi_{J'}$ is  connected, and that it must be a subgraph of $D_k$ for some $1\leq k\leq \ell$.
For each $1\leq k\leq \ell$, let
\[\mathcal{J}_{conn}^k(P)=\{J\in \mathcal{J}_{conn}(P)\mid \text{ the graph } \chi_J \text{ is a subgraph of } D_k\}.\]
In particular, if $J'=\Lambda_i^P\in V(D_k)$ is a principal order ideal, then the graph $\chi_{J'}$ has only one vertex $\Lambda_i^P$, thus $\chi_{J'}$ is of course a subgraph of $D_k$.
It follows that $V(D_k)\subseteq \mathcal{J}^k_{conn}(P)$ for each $1\leq k\leq \ell$.  It is clear that
\[\mathcal{J}_{conn}(P)=\mathcal{J}^1_{conn}(P)\uplus\mathcal{J}^2_{conn}(P)\uplus\cdots\uplus\mathcal{J}^\ell_{conn}(P).\]
For each $1\leq k\leq \ell$, let $C_k$ be the connected component of $G_P$ such that $D_k$ is a subgraph of $C_k$ (it turns out that for each $D_k$, there exists a unique $C_k$ such that $D_k$ is a subgraph of $C_k$). We proceed to show that $V(C_k)\subseteq \mathcal{J}^k_{conn}(P)$. Note that if $J_a\in \mathcal{J}^s_{conn}(P)$ and $J_b\in \mathcal{J}^t_{conn}(P)$ for some $s\neq t$, the first assertion of  Lemma \ref{intersect trivially} tells us that $J_a$ and $J_b$ are not adjacent in $G_P$.
Thus, by the connectivity of $C_k$ in $G_P$, all members of $V(C_k)$ must belong to $\mathcal{J}^k_{conn}(P)$ since we already have $V(D_k)\subseteq \mathcal{J}^k_{conn}(P)$. And then, we get that
$V(D_k)\subseteq V(C_k)\subseteq \mathcal{J}^k_{conn}(P).$
That is to say, for any  $J'\in V(C_k)$, the graph $\chi_{J'}$ is a subgraph of $D_k$. Therefore, $J'\subseteq \bigcup_{\Lambda_i^P\in V(D_k)}\Lambda_i^P$  for any  $J'\in V(C_k)$. This leads to the following equality:
\begin{equation}\label{max J equal}
 J_k^{max}=\bigcup_{J'\in V(C_k)}J'=\bigcup_{\Lambda_i^P\in V(D_k)}\Lambda_i^P.
\end{equation}

For the given $J$,  we assume that $\chi_J$ is a subgraph of the connected component $D_r$  of $H_P$ for some $1\leq r\leq \ell$,  and then $D_r$ is  a subgraph of $C_r$. Thus in view of \eqref{max J equal},  when $J\neq J_r^{max}$, it follows that $J\neq \bigcup_{\Lambda_i^P\in V(D_r)}\Lambda_i^P$.
By the second assertion of Lemma \ref{intersect trivially}, in the graph $G_P$ we see that  $J$ is adjacent to some vertex of $D_r$,  therefore, $J$ is also a vertex of $C_r$.
\end{proof}

We are almost ready for the proof of Theorem \ref{def-descent}. Note that the definition of $\mathrm{Des}(M)$ ($M\in\mathscr{M}(G_P)$) is indirect, which uses the map $\Psi$ from $\mathscr{M}(G_P)$ to $\mathscr{F}(P)$. In order to make the proof of Theorem \ref{def-descent} more clear, we shall give another characterization of $\mathrm{Des}(M)$ which only uses the information of $M$.
Before doing this, we shall introduce one more notation. Given $J_a,J_b\in M$, we say that $J_a\prec_M J_b$ if $J_a\subset J_b$ and there exists no  $J\in M$ such that $J_a\subset J\subset J_b$. Our new characterization of $\mathrm{Des}(M)$ is as follows.

\begin{lem}\label{descent explaining}
Given $M\in\mathscr{M}(G_P)$, then $i\in \mathrm{Des}(M)$ if and only if there exists $j<i$ such that
$J_a\prec_M J_b$, where $J_a,J_b\in M$ are connected order ideals uniquely determined by $i,j$ respectively as in Lemma \ref{the principal ideal of F}.
\end{lem}
\begin{proof}
By definition, $i\in \mathrm{Des}(M)=\mathrm{Des}(F_M)$ if and only if the parent of $i$, say $j$, is greater than $i$ with respect to the natural order on integers. Recall that if $j$ is the parent of $i$, then $i<_{F_M}j$ and there exists no $k$ such that $i<_{F_M}k<_{F_M}i$. It follows from Lemma \ref{the principal ideal of F} that there exist
two connected order ideals $J_a, J_b$ in $M$ satisfying $J_a\setminus \mu(M,J_a)=\{i\}, J_b\setminus \mu(M,J_b)=\{j\}$.
By the construction of $F_M$, we have $J_a\subset J_b$ but there exists no $J\in M$ such that $J_a\subset J\subset J_b$, namely $J_a\prec_M J_b$.
\end{proof}

As shown above, the  relation $\prec_M$ plays an important role for the new characterization of $\mathrm{Des}(M)$.
To prove Theorem \ref{def-descent}, we also need the following lemma, which is evident by definition. Recall that the set $U_{max}(M,J)$  is defined by \eqref{def-UMJ}.
\begin{lem}\label{cover in M}
Given $J_a,J_b\in M$, if $J_a\prec_M J_b$ then $J_a\in U_{max}(M,J_b)$.
\end{lem}

Now we are in the position to prove Theorem \ref{def-descent}. From now on we shall assume that $P$ is naturally labeled.

\noindent \textit{Proof of Theorem \ref{def-descent}.}
There are two cases to consider.

(1). The connected component $C_r$ has only one vertex, say $J_r$. Thus $M_r$ can only be the unique one maximum independent set $\{J_r\}$ of  $C_r$.
By Lemma \ref{the principal ideal of F}, we have $J_r\setminus\mu(M^1, J_r)=\{i\}$
for some $i\in \{1,2,\ldots,n\}$.
In this case, we first prove that
\begin{equation}\label{emptyset isolate vertex}
  \mathrm{Des}(M_r,M^1)=\mathrm{Des}(M_r,M^2)=\emptyset.
\end{equation}
Otherwise, suppose that $\mathrm{Des}(M_r,M^1)=\{i\}$. By the definition of $\mathrm{Des}(M_r,M^1)$, we have $i\in \mathrm{Des}(M^1)$.
By Lemma \ref{descent explaining}, there exist $j<i$ and $J\in M^1$ such that $J\setminus \mu(M^1,J)=\{j\}$ and $J_r\prec_{M^1} J$.

We proceed to show that it is impossible to have such a pair $(i,j)$.
Let us consider the order relation between $i$ and $j$ in the poset $P$. It cannot be $j<_P i$,  since $i\in J_r\subset J$  and Lemma \ref{the independent set K} tells us that $j$ is a maximal element of $J$. Then it might be $i<_P j$, or $i$ and $j$ are incomparable in $P$. Since $P$ is naturally labeled and $j<i$, it can not be $i<_P j$. Suppose that $i$ and $j$ are incomparable in $P$. Since $J_r\setminus\mu(M^1, J_r)=\{i\}$, it follows from  Lemma \ref{the independent set K} that $i$ is a maximal element of $J_r$. We  proceed to prove that $i$ is also a maximal elements of $J$.
To see this, it is enough to show that there exists no $k\in J$ satisfying $i<_P k$. Note that
\[J=\{j\}\cup \mu(M^1,J)=\{j\}\cup\left(\bigcup_{J'\in U(M^1,J)}J'\right)=\{j\}\cup \left(\bigcup_{J'\in U_{max}(M^1,J)}J'\right).\]
By Lemma \ref{cover in M}, the relation $J_r\prec_{M^1} J$ implies that $J_r\in U_{max}(M^1,J)$. Then there are three cases to consider:
\begin{itemize}
\item[(i)] If $k=j$, then $i$ and $k$ are incomparable in $P$;

\item[(ii)] If $k\in J_r$, in this case we have $k\leq_P i$, or $i$ and $k$ are incomparable in $P$, because $i$ is a maximal element of $J_r$;

\item[(iii)] If $k\in J'$ for some $J'\in U_{max}(M^1,J)$ but $J'\neq J_r$, we obtain that $i$ and $k$ are incomparable in $P$,  since by Lemma \ref{lemm-empty} we have $J'\cap J_r=\emptyset$, which implies that for any $u\in J_r$, $v\in J'$, $u$ and $v$ are incomparable in $P$.
\end{itemize}
Hence there exists no $k\in J$ such that $i<_P k$, i.e., $i$ is a maximal element of $J$.
It follows that $\{i,j\}\subseteq gs(J)$ and then the graphs $\chi_{J_r}$ and $\chi_{J}$ have a common vertex $\Lambda_i^P$. Then by Lemma \ref{induced graph}, the graphs $\chi_{J_r}$ and $\chi_{J}$ belong to the same connected component $C_s$ of $G_P$.
Hence $C_s$ has at least two vertices $\Lambda_i^P$ and $\Lambda_j^P$. By Lemma \ref{isolated vertex} and the hypothesis that $J_r$ is an isolated vertex of $G_P$,  we obtain $J_r=\bigcup_{J'\in V(C_s)}J'$ and $J\subseteq \bigcup_{J'\in V(C_s)}J'$. This contradicts with the assumption that $J_r\prec_{M^1} J$. Hence  $i$ and $j$ cannot be incomparable in $P$, a contradiction.

Since such a pair $(i,j)$ can not exist, it follows that $\mathrm{Des}(M_r,M^1)=\emptyset$. By using a similar argument, one can also prove that  $\mathrm{Des}(M_r,M^2)=\emptyset$.
Moreover, by the definition of  $\overline{\mathrm{Des}}(M_r,M)$,  it is clear that
\[ \overline{\mathrm{Des}}(M_r,M^1)=\overline{\mathrm{Des}}(M_r,M^2)=\emptyset.\]

(2). $C_r$ has at least two vertices. In this case, $M_r\subset V(C_r)$. By Lemma \ref{isolated vertex}, we see that  $J_r^{max}=\bigcup_{J'\in V(C_r)}J'$ is an isolated vertex of $G_P$. Hence  $J_r^{max}\in M$ holds for any maximum independent set of $G_P$, and in particular $J_r^{max}\in M^1$ as well as $J_r^{max}\in M^2$.

We first prove that for any $J\in M_r$ or $J=J_r^{max}$, 
\begin{equation}\label{J M_J }
J\setminus \mu(M^1,J)=J\setminus\mu(M^2,J).
\end{equation}
To see this, we partition the set $U(M^2,J)$ into two subsets $B_1$ and $B_2$, where
\begin{eqnarray*}
  B_1 &=& \{ J_1\in U(M^2,J)\mid J_1\in  V(C_r)\},\\[5pt]
  B_2 &=& \{ J_2\in U(M^2,J)\mid J_2\notin V(C_r)\}.
\end{eqnarray*}
Assume $J\setminus \mu(M^1,J)=\{j\}$.
We claim that  $j\notin J_2$
for any $J_2\in B_2$. Otherwise, suppose to the contrary that  there exists some $J_2\in B_2$ such that $j\in J_2$.
It follows from Lemma \ref{the independent set K} that $j\in gs (J)$.
On the other hand, since $J_2\subset J$, we obtain that $j\in gs(J_2)$.
Hence the graph $\chi_J$ and $\chi_{J_2}$ have a common vertex $\Lambda_j^P$. Then by Lemma \ref{induced graph} the graphs $\chi_J$ and $\chi_{J_2}$ belong to the same connected component of $G_P$. We proceed to show that $\chi_{J_2}$ is a subgraph of $C_r$. To see this, there are two cases to consider.
\begin{itemize}
  \item [(i)] Suppose that $J\in M_r\subset V(C_r)$ (then $J\neq J_r^{max}$), namely, $J$ is a vertex of the connected component $C_r$. It follows from the second assertion of Lemma \ref{isolated vertex} that $\chi_J$ and $J$ are contained in the same connected component $C_r$ of $G_P$. Hence both  $\chi_J$ and $\chi_{J_2}$ are subgraphs of $C_r$.

  \item [(ii)] Suppose that $J=J_r^{max}=\bigcup_{J'\in V(C_r)}J'$. Let $i\in gs(J)$ be a maximal element of $J$, then  there exists some $J'\in V(C_r)$ such that $i\in J'$. It follows that $i$ is also a maximal element of $J'$, namely,  $i\in gs(J')$. Hence the graphs  $\chi_J$ and $\chi_{J'}$ have at least one common vertex $\Lambda_i^P$, and then  $\chi_J$ and $\chi_{J'}$ belong to the same connected component of $G_P$.  The second assertion of Lemma \ref{isolated vertex} tells us that for any $J'\in V(C_r)$,   $\chi_{J'}$ and $J'$ are contained in the same connected component $C_r$ of $G_P$. Hence $\chi_J, \chi_{J'}$ and $\chi_{J_2}$ are all subgraphs of $C_r$.
\end{itemize}
On the other hand, because $J_2\subset J$, we have $J_2\neq J_r^{max}$. Then by the second assertion of Lemma \ref{isolated vertex} we get  $J_2\in V(C_r)$, leading to a contradiction. Hence the claim, that $j\notin J_2$ for any $J_2\in B_2$,  is true.

Recall that $M^1\cap V(C_r)=M^2\cap V(C_r)=M_r.$ It is routine to verify that
\[U(M^1,J)\cap M_r=U(M^2,J)\cap M_r=B_1,\]
Combining \eqref{eqn-umu} and the above identity, we get that
\[j\in J\setminus  \mu(M^1,J)\subseteq J\setminus\bigcup_{J_1\in B_1}J_1.\]
As we have shown that $j\notin J_2$ for any $J_2\in B_2$, so again by \eqref{eqn-umu} there holds
\[j\in J\setminus\bigcup_{J'\in (B_1\cup B_2)}J'=J\setminus\bigcup_{J'\in U(M^2,J)}J'=J\setminus \mu(M^2,J).\]
Thus, by Lemma \ref{the independent set K}, the set $J\setminus \mu(M^2,J)$ contains  exactly one element, which can only be $j$.  Therefore, we have
\[\{j\}=J\setminus \mu(M^2,J)=J\setminus \mu(M^1,J).\]

We proceed to show that $\mathrm{Des}( M_r,M^1)\subseteq \mathrm{Des}( M_r,M^2)$.
Let $i\in \mathrm{Des}( M_r,M^1)$, and by  the definition of  $\mathrm{Des}( M_r,M^1)$ and Lemma \ref{the independent set K} there exists  $J_a\in M_r$ such that $J_a\setminus \mu(M^1,J_a)=\{i\}$. By Lemma \ref{descent explaining}, there exist $j<i$ and $J_b\in M^1$ such that $J_b\setminus \mu(M^1,J_b)=\{j\}$ and $J_a\prec_{M^1} J_b$. We claim that $J_b\in V(C_r)$ or $J_b=J_r^{max}$. Suppose otherwise that $J_b$ is not a vertex of $C_r$ and  $J_b\neq J_r^{max}$. Since $J_a\in V(C_r)$ and $J_a\subset J_b$, it follows from Lemma \ref{absorbtion} that  $J'\subset J_b$ for any  $J'\in V(C_r)$. Hence $J_r^{max}\subset J_b$.  Thus we obtain $J_a\subset J_r^{max}\subset J_b$. Recall that $J_r^{max}\in M^1$, the relation $J_a\subset J_r^{max}\subset J_b$ contradicts the assumption that $J_a\prec_{M^1}J_b$. Recall also that we have shown  $J_r^{max}\in M^2$. If $J_b=J_r^{max}$ then $J_b\in M^2$. If $J_b\in V(C_r)$, then $J_b\in M_r=M^2\cap V(C_r)$, and hence also $J_b\in M^2$. We further show that $J_a\prec_{M^2} J_b$. Otherwise, suppose there exists some $ J_c\in M^2$ such that $J_a\subset J_c\subset J_b$. By the hypothesis that $J_a\prec_{M^1}J_b$ and $M^1\cap V(C_r)=M^2\cap V(C_r)=M_r$, it follows that $J_c\notin M_r\subset V(C_r)$. Then by Lemma \ref{absorbtion},  for any $J'\in V(C_r)$, there is $J'\subset J_c$. Hence $J_b\subseteq \bigcup_{J'\in V(C_r)}\subset J_c$, leading to a contradiction. Thus, for any $i\in \mathrm{Des}( M_r,M^1)$,  by \eqref{J M_J }  there exist  $J_a,J_b\in M^2$  such that  $J_a\setminus \mu(M^2,J_a)=\{i\}$, $J_b\setminus \mu(M^2,J_b)=\{j\}$, $J_a\prec_{M^2}J_b$ and $i>j$.  This means $i\in \mathrm{Des}( M_r,M^2)$ for any $i\in \mathrm{Des}( M_r,M^1)$. Hence $\mathrm{Des}( M_r,M^1)\subseteq \mathrm{Des}( M_r,M^2)$.

It can be proved in a similar way that $\mathrm{Des}( M_r,M^2)\subseteq \mathrm{Des}( M_r,M^1)$. So we get $\mathrm{Des}( M_r,M^1)=\mathrm{Des}( M_r,M^2).$ Combining this and \eqref{J M_J }, we further obtain $\overline{\mathrm{Des}}( M_r,M^1)=\overline{\mathrm{Des}}( M_r,M^2),$
as desired.
\qed

We proceed to prove Theorem \ref{the seconde main result}.

\noindent \textit{Proof of Theorem \ref{the seconde main result}.}
Given a maximum independent set $M$ of $G_P$, let
\[\overline{\mathrm{Des}}(M)=\big\{J\in M \mid  J\setminus\mu(M,J)=\{i\} \text{ for some } i\in \mathrm{Des}(M)\big\}.\]
Recall that $\mathscr{M}(C_r)$ is the set of maximum independent sets of $C_r$ for each $1\leq r\leq h$, respectively.
It is clear that $M$ admits the following natural decomposition:
$$M=M_1\uplus M_2\uplus\cdots\uplus M_h, \text{ where } M_r\in \mathscr{M}(C_r).$$
It follows from Theorem \ref{def-descent} that  both $\mathrm{Des}(M_r)$ and $\overline{\mathrm{Des}}(M_r)$ are well-defined, and hence
\begin{align}
  \mathrm{Des}(M)&=\mathrm{Des}(M_1)\uplus\mathrm{Des}(M_2)\uplus\cdots\uplus\mathrm{Des}(M_h), \label{the union descent set of M}\\
  \overline{\mathrm{Des}}(M)&=\overline{\mathrm{Des}}(M_1)\uplus\overline{\mathrm{Des}}(M_2)\uplus\cdots\uplus\overline{\mathrm{Des}}(M_h).\label{the union descent set of over  M}
\end{align}
Thus, by \eqref{F_P_FR}, Theorem \ref{The First Main Result} and Lemma \ref{the principal ideal of F}, we get that
\begin{eqnarray*}
  F_P(\textbf{x}) &=&   \sum_{M\in\mathscr{M}(G_P)}\frac{\prod_{J\in \overline{\mathrm{Des}}(M)}\prod_{k\in J}x_k}{\prod_{J\in M}(1-\prod_{\ell\in J}x_\ell)}.
\end{eqnarray*}
By \eqref{the union descent set of over  M}, we then have
\begin{eqnarray*}
  F_P(\textbf{x}) &=&  \sum_{M_1\in \mathscr{M}(C_1)} \sum_{M_2\in \mathscr{M}(C_2)}\cdots  \sum_{M_h\in \mathscr{M}(C_h)}\frac{\prod_{r=1}^h\prod_{J\in \overline{ \mathrm{Des}}(M_r)}\prod_{k\in J}x_k}{\prod_{r=1}^h\prod_{J\in M_r}(1-\prod_{\ell\in J}x_\ell)}\\
   &=&  \prod_{r=1}^h \sum_{M_r \in \mathscr{M}(C_r) }\frac{\prod_{J\in \overline{\mathrm{Des}}(M_r)}\prod_{k\in J}x_k}{\prod_{J\in M_r}(1-\prod_{\ell\in J}x_\ell)}. ~~~~~~~~~~~~~~~~~~~~~~~~~~~~~~~~~~~\qed
\end{eqnarray*}

We would like to point out that Theorem \ref{the seconde main result} enables us to give an alternative proof to  F\'eray and Reiner's formula  \eqref{formula for forests with duplications}. To this end, let $P$ be a naturally labeled forest with duplications as defined by F\'eray and Reiner \cite{FerayReiner}, namely,  for any connected order ideal $J_a$ of $P$, there exists at most one other connected order ideal $J_b$ such that $J_a$ and $J_b$  intersect nontrivially.
Assume that $G_P$ has $h$ connected components $C_1,C_2,\ldots,C_h$.
Then each $C_r$ has at most two vertices, and hence each connected component of $H_P$  has also at most two vertices.

We claim that when a connected component $C$ of $G_P$ has two vertices, say $J_a$ and $J_b$, then both $J_a$ and $J_b$ are principal order ideals of $P$. Otherwise, suppose that $J_a$ is not a principal order ideal of $P$. Then the graph $\chi_{J_a}$ has more than one vertices. Recall that $\chi_{J_a}$ is a subgraph of $H_P$. By Lemma \ref{induced graph} and the fact that  each connected component of the graph $H_P$ has at most two vertices, the graph $\chi_{J_a}$ is a connected component of $H_P$.
It then follows from  \eqref{max J equal}  and the first assertion of Lemma \ref{isolated vertex} that $J_a$ is an isolated vertex of $G_P$, a contradiction. Similarly, $J_b$ is also a principal order ideal of $P$.

Therefore, we may assume that for $1\leq r\leq d$ the component  $C_r$ has two vertices (both of them are principal order ideals of $P$), say $\Lambda_{i_r}^P$ and $\Lambda_{j_r}^P$, and for $d< r\leq h$ the component $C_r$ has only one vertex. Thus, for $1\leq r\leq d$,
there are two choices for $ M_r$, namely, $M_r=\{\Lambda_{i_r}^P\}$ or $M_r=\{\Lambda_{j_r}^P\}.$
We assume that $i_r>j_r$. Then
\[\overline{\mathrm{Des}}(\{\Lambda_{i_r}^P\})=\Lambda_{i_r}^P,~~\overline{\mathrm{Des}}(\{\Lambda_{j_r}^P\})=\emptyset.\]
For $d<r\leq h$, let $J_r$ be the only vertex   of $C_r$, and then $\overline{\mathrm{Des}}(\{J_r\})=\emptyset$.
By Theorem \ref{the seconde main result}, we obtain that
\begin{eqnarray*}\label{specialize for forests with duplications}
  \nonumber F_P(\textbf{x})  &=& \prod_{1\leq r\leq d}\left[\frac{\textbf{x}^{\Lambda_{i_r}^P}}{\left(1-\textbf{x}^{\Lambda_{i_r}^P}\right)}
  +\frac{1}{\left(1-\textbf{x}^{\Lambda_{j_r}^P}\right)}\right]\prod_{d< r\leq h}\frac{1}{\left(1-\textbf{x}^{J_r}\right)}\\
   &=& \prod_{1\leq r\leq d}\left[\frac{1-\textbf{x}^{\Lambda_{i_r}^P}\textbf{x}^{\Lambda_{j_r}^P}}
   {\left(1-\textbf{x}^{\Lambda_{i_r}^P}\right)
   \left(1-\textbf{x}^{\Lambda_{j_r}^P}\right)}\right]
   \prod_{d<r\leq h}\frac{1}{\left(1-\textbf{x}^{J_r}\right)},
\end{eqnarray*}
where  $\textbf{x}^A=\prod_{i\in A}x_i$ for  a subset $A\subseteq \{1,2,\ldots,n\}$. It is straightforward to verify that the above formula is equivalent to \eqref{formula for forests with duplications}.

\section{Counting linear extensions}\label{sect-application}
In this section, we take an example to show that formula \eqref{the product formula for F_P} can be used to derive the generating function of major index of linear extensions of $P$, as well as to count the number $|\mathcal{L}(P)|$ of linear extensions of $P$.

The generating function $F_P(q)$ of major index of linear extensions of $P$ is denoted by $F_P(q)=\sum_{w\in\mathcal{L}(P)}q^{\mathrm{maj}(w)},$
where $\mathrm{maj}(w)=\sum_{i\in\mathrm{Des}(w)}i$ is called the major index of $w$. By letting  $x_1=\cdots=x_n=q$ respectively in  \eqref{Stanley-formula for F_P} and \eqref{the product formula for F_P}, we are led to the following identity
\begin{equation}\label{q-major index}
  F_P(q)=[n]!_q\prod_{r=1}^h\sum_{ M_r\in \mathscr{M}(C_r)}\frac{q^{\sum_{J\in \overline{\mathrm{Des}}(M_r)}|J|}}{\prod_{J\in M_r} [|J|]_q},
\end{equation}
where $M_r$ ranges over maximum independent sets of $C_r$, $[i]_q=1-q^i$ for any $i$ and $[m]!_q=\prod_{i=1}^m [i]_q$.

Moreover, when  $q$ tends  to $1$ on both sides of \eqref{q-major index}, we
arrive at the following formula for the number of linear extensions of $P$:
\begin{equation}\label{count linear extensions}
  |\mathcal{L}(P)|=n!\prod_{r=1}^h\sum_{ M_r\in \mathscr{M}(C_r)}\frac{1}{\prod_{J\in M_r}|J|}.
\end{equation}
Note that the number of linear extensions of $P$ is independent of the labelling of $P$. Thus formula \eqref{count linear extensions} is also valid  in the cases when $P$ is not naturally labeled.

We would like to mention that calculating the number of linear extensions  for  general posets  has been proved to be a $\sharp P$-hard problem by Brightwell and Winkler \cite{BrightwellWinkler}. However, in the case when $P$ is a poset such that  each connected component $C_r$ of $G_P$ has small size of vertex set, we shall illustrate that    formula \eqref{count linear extensions} provides an efficient way to count the number  of linear extensions of $P$.
For example, take the naturally labeled poset $P$  in Figure \ref{The poset P}.
 From the graph of $G_P$ as illustrated in Figure \ref{The graph G_P}, we obtain that
\begin{enumerate}
  \item For the connected component $C_1$, there are $6$ choices  for $ M_1$:

  \begin{tabular}{|c|c|c|c|c|}
    \hline
    $M_1$ & $\{\Lambda_4^P,\Lambda^P_{4,5}\}$  &$\{\Lambda_4^P,\Lambda_{4,6}^P\}$  &  $ \{\Lambda_5^P,\Lambda_{4,5}^P\}$ &  $\{\Lambda_5^P,\Lambda_{5,6}^P\}$  \\
    \hline
     $\mathrm{Des}(M_1)$& $\emptyset$ & $\{6\}$& \{5\} &  \{6\}    \\
    \hline
     $\overline{\mathrm{Des}}(M_1)$& $\emptyset$ &$\{\Lambda_{4,6}^P\}$  & $\{\Lambda_5^P\}$  &  $\{\Lambda_{5,6}^P\}$ \\
    \hline
  \end{tabular}

   \begin{tabular}{|c|c|c|}
    \hline
    $M_1$ & $\{\Lambda_6^P,\Lambda_{4,6}^P\}$  & $\{\Lambda_6^P,\Lambda_{5,6}^P\}$ \\
    \hline
     $\mathrm{Des}(M_1)$&  \{6\}  &\{5,6\}  \\
    \hline
     $\overline{\mathrm{Des}}(M_1)$& $\{\Lambda_6^P\}$  &  $\{\Lambda_6^P,\Lambda_{5,6}^P\}$  \\
    \hline
  \end{tabular}

  \item For the connected component $C_2$, there are $5$ choices  for $ M_2$:

    \begin{tabular}{|c|c|c|c|}
    \hline
    $M_2$ & $\{\Lambda_{10}^P,\Lambda_{15}^P,\Lambda_{13,15}^P\}$  &  $\{\Lambda_{10}^P,\Lambda_{10,13}^P,\Lambda_{14}^P\}$ &  $\{\Lambda_{10}^P,\Lambda_{10,13}^P,\Lambda_{13,15}^P\}$
        \\
    \hline
     $\mathrm{Des}(M_2)$& $\{15\}$ & $\emptyset$  &  $\{15\}$   \\
    \hline
     $\overline{\mathrm{Des}}(M_2)$& $\{\Lambda_{15}^P\}$  &  $\emptyset$ & $\{\Lambda_{13,15}^P\}$     \\
    \hline
  \end{tabular}

  \begin{tabular}{|c|c|c|}
    \hline
    $M_2$ &  $\{\Lambda_{13}^P,\Lambda_{10,13}^P,\Lambda_{14}^P\}$ &  $ \{\Lambda_{13}^P,\Lambda_{10,13}^P,\Lambda_{13,15}^P\}$ \\
    \hline
     $\mathrm{Des}(M_2)$& $\{13\}$  &   $\{13,15\}$    \\
    \hline
     $\overline{\mathrm{Des}}(M_2)$& $\{\Lambda_{13}^P\}$ &  $\{\Lambda_{13}^P,\Lambda_{13,15}^P\}$     \\
    \hline
  \end{tabular}

  \item  For  the connected component $C_3$, there are $3$ choices  for $M_3$:

  \begin{tabular}{|c|c|c|c|}
    \hline
    $M_3$ &$\{\Lambda_{11}^P,\Lambda_{11,9}^P\}$  &  $\{\Lambda_{9}^P,\Lambda_{11,9}^P\}$&  $\{\Lambda_{9}^P,\Lambda_{12}^P\}$
        \\
    \hline
     $\mathrm{Des}(M_3)$& $\{11\}$ & $\emptyset$ &  $\{12\}$   \\
    \hline
     $\overline{\mathrm{Des}}(M_3)$& $\{\Lambda_{11}^P\}$ & $\emptyset$ & $\{\Lambda_{12}^P\}$    \\
    \hline
  \end{tabular}

  \item For the connected component  $C_4$, there are $2$ choices for $ M_4$:

   \begin{tabular}{|c|c|c|}
    \hline
    $M_4$ &   $\{\Lambda_{16}^P\}$ &  $\{\Lambda_{17}^P\}$ \\
    \hline
     $\mathrm{Des}(M_4)$& $\emptyset$ &   $\{17\}$    \\
    \hline
     $\overline{\mathrm{Des}}(M_4)$& $\emptyset$ &  $\{\Lambda_{17}^P\}$      \\
    \hline
  \end{tabular}

  \item For connected  components which have only one vertex,  each of them has only one choice for each $M_r$,  and $\mathrm{Des}(M_r)=\emptyset$ as well as $\overline{\mathrm{Des}}(M_r)=\emptyset$.
\end{enumerate}
Therefore, invoking  formula \eqref{q-major index}, we see that $ F_P(q)=\sum_{w\in\mathcal{L}(P)}q^{\mathrm{maj}(w)}$ equals
\begin{eqnarray*}
{[17]!}_q \bigg[\frac{1}{{[6]}_q}\left( \frac{1+2q^3+2q^5+q^8}{{[3]}_q{[5]}_q}\right)\bigg]\bigg[\frac{1}{{[15]}_q}\bigg(\frac{q^{13}+1+q^{14}}{{[7]}_q{[13]}_q{[14]}_q}+\frac{q^{12}+q^{26}}{{[12]}_q{[13]}_q{[14]}_q}\bigg)\bigg]  \\
   \times \bigg[\frac{1}{{[5]}_q}\bigg(\frac{q^3}{{[3]}_q{[4]}_q}+\frac{1}{{[2]}_q{[4]}_q}+\frac{q^3}{{[2]}_q{[3]}_q}\bigg)\bigg]\bigg[\frac{1}{{[17]}_q}\frac{(1+q^{16})}{{[16]}_q}\bigg]\times 1^5.
\end{eqnarray*}
Letting $q\rightarrow 1$ in the above formula, we arrive at
 \begin{eqnarray*}
 |\mathcal{L}(P)| &=& 17!\times\left( \frac{1}{6}\times\frac{6}{3\times 5}\right)\times\left[\frac{1} {15}\times\left(\frac{3}{7\times 13\times 14}+\frac{2}{13\times 12\times 14}\right)\right] \\[5pt]
   &&\ \  \times\left[\frac{1}{5}\times\left(\frac{1}{3\times 4}+\frac{1}{3\times 2}+\frac{1}{4\times 2}\right)\right] \times \left( \frac{1}{17}\times\frac{2}{16}\right)\times 1^5 \\[5pt]
   &=&  2851200.
\end{eqnarray*}
This coincides with the result by listing all linear extensions by using Sage \cite{Sage}.

\vskip 3mm \noindent {\bf Acknowledgments.}
Our deepest gratitude goes to the anonymous reviewer(s) for their careful work and thoughtful suggestions that have helped improve this paper substantially.
This work was supported by the PCSIRT Project of the Ministry of Education and the National Science Foundation of China.

\end{document}